\documentclass[12pt]{amsart}

\usepackage{amsmath,graphicx,verbatim,amsfonts,amssymb,amsthm,comment,amscd,mathtools}
\usepackage{xcolor}
\usepackage{soul}
\usepackage{subcaption, enumerate}
\usepackage{graphics}
\usepackage[section]{placeins}
\usepackage{enumitem}
\usepackage{hyperref}
\hypersetup{
    colorlinks=true, 
    linktoc=all,     
    linkcolor=blue,  
    linktocpage=false,
}

\newtheorem{thm}{Theorem}[section]

\newtheorem{prop}[thm]{Proposition}
\newtheorem{lem}[thm]{Lemma}

\newtheorem{lemma}[thm]{Lemma}

\theoremstyle{remark}

\newtheorem{remark}[thm]{Remark}

\numberwithin{equation}{section}

\newcommand{\R}{\mathbb{R}}
\newcommand{\Z}{\mathbb{Z}}
\newcommand{\C}{\mathbb{C}}
\newcommand{\N}{\mathbb{N}}

\newcommand{\old}[1]{}
\newcommand{\disf}[1][F]{\circ_F}

\newcommand{\cna}{c_{\uparrow}}
\newcommand{\csa}{c_{\downarrow}}
\newcommand{\cea}{c_{\rightarrow}}
\newcommand{\cwa}{c_{\leftarrow}}

\newcommand{\ptran}[2]{p_{#1\to #2}}
\newcommand{\ctran}[2]{c_{#1\to #2}}

\newcommand{\wcr}{w_{+}}
\newcommand{\vcr}{v_{+}}
\newcommand{\scr}{S(\wcr)}

\title{The limit shape of the Leaky Abelian Sandpile Model}

\author[Ian Alevy]{Ian Alevy}
\email{ian.alevy@rochester.edu}
\address{University of Rochester\\ Rochester, NY}

\author[Sevak Mkrtchyan]{Sevak Mkrtchyan}
\email{sevak.mkrtchyan@rochester.edu}
\address{University of Rochester\\ Rochester, NY}

\date{\today}

\begin{document}

\begin{abstract}
    The leaky abelian sandpile model (Leaky-ASM) is a growth model in which $n$ grains of sand start at the origin in $\mathbb{Z}^2$ and diffuse along the vertices according to a toppling rule. A site can topple if its amount of sand is above a threshold. In each topple a site sends some sand to each neighbor and leaks a portion $1-1/d$ of its sand. 

    We compute the limit shape as a function of $d$ in the symmetric case where each topple sends an equal amount of sand to each neighbor. The limit shape converges to a circle as $d\to 1$ and a diamond as $d\to\infty$. We compute the limit shape by comparing the odometer function at a site to the probability that a killed random walk dies at that site. 
    
    When $d\to 1$ the Leaky-ASM converges to the abelian sandpile model (ASM) with a modified initial configuration. 
    We also prove the limit shape is a circle when simultaneously with $n\to\infty$ we have that $d=d_n$ converges to $1$ slower than any power of $n$. To gain information about the ASM faster convergence is necessary.
\end{abstract}

\maketitle
\tableofcontents

\section{Introduction}

The \emph{Abelian Sandpile Model} (ASM) is a cellular automaton defined on the lattice \(\Z^2\). 
The input is a \emph{sandpile configuration} \(s:\Z^2\to \N\) which represents the number of chips or grains of sand at each site $x\in\Z^2$. The sandpile $s$ evolves under the following rule: If a site \(x\) has at least 4 chips then it ``fires'' or ``topples,'' giving one chip each to each of its \(4\) nearest-neighbors (north, east, south, and west). 
The sandpile evolves until there are no more sites that can topple. The word ``Abelian'' in the name of the model refers to the fact that the final stable configuration does not depend on the order in which sites topple. The ASM model was introduced by Per Bak, Chao Tang and Kurt Wiesenfeld in 1987 \cite{BTW-ASM1987} and has been extensively studied in the decades since. A survey of the ASM and its connection to the router-router model can be found in \cite{hl2008}. See \cite{lp2010-notices}, \cite{ds2013}, and \cite{lp2017} for current overviews of the research on the ASM. A related model was studied earlier in the mathematical literature \cite{df1991}.

The ASM has been studied extensively in the case when the initial configuration consists of \(n\) chips at the origin and zero chips everywhere else. It is known that the stable configuration has a scaling limit \cite{ps2013} which is bounded between circles of radii \(c_1 \sqrt{n}\) and \(c_2 \sqrt{n}\) \cite{lp2009} (see \cite{lp2010} for a similar statement when the initial configuration has multiple point sources).
Simulations show the emergence of a fractal structure in the limit shape (see Figure \ref{fig:ASM} for a simulation of the stable configuration with \(n=10^7\) chips). It is known that the boundary of the limit shape is a Lipschitz graph \cite{as2019}. The fractal structure and the local patterns have been studied in \cite{lps2016} and \cite{ps2020}.
However many mathematical questions about the limit shape remain unanswered. 
Simulations suggest that the limit shape is convex and that the boundary has flat regions, but to date no mathematical explanation has been given for either claim. It has not even been shown that the boundary is not a circle. 

\begin{figure}
    \centering
        \includegraphics[width=9cm]{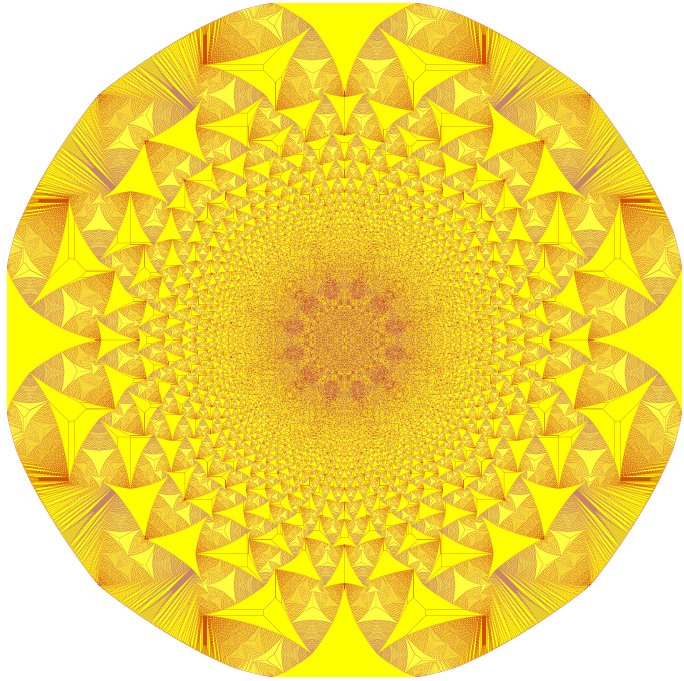}
        \caption{Standard ASM with $10^7$ chips}
        \label{fig:ASM}
\end{figure}

In this paper we obtain exact limit shape results for a one parameter deformation of the ASM, called the \emph{Leaky Abelian Sandpile Model} (Leaky-ASM), in which dissipation is present. Dissipative sandpile models were introduced in \cite{mkk1990}. We give an explicit connection between the set of visited sites for the Leaky-ASM and the death probabilities for the killed random walk. We then obtain the limit shape of the uniform Leaky-ASM by analyzing the asymptotic death probabilities of the associated killed random walk. Our model has a spectral curve similar to the spectral curve analyzed in \cite{kw2020}.

The Leaky-ASM is a generalization of the standard ASM in three key ways. First, sandpiles are allowed to take real values, \(s:\Z^2\to \R\), and are not assumed to be integer. Second, each topple may send a different number of chips in each direction. Let \(\cna, \cea, \csa, \cwa \in \R_{\geq 0}\) be non-negative real numbers. Each time a site topples \(\cna, \cea, \csa\), and \(\cwa\) chips are sent in the north, east, south, and west directions. Thirdly, each time a site topples it ``leaks'' chips. Let \(d\in \R\) with \(d>1\) and let \(c=\cna+\cea+\csa+\cwa\). A site topples whenever it has more than $cd$ chips and whenever it topples it loses \(cd\) chips, \(c\) of which are distributed to the nearest-neighbors and the remaining \(cd -c = c(d-1)\) chips leak. We call \(d-1\) the \emph{leakiness} parameter. Chips which leak disappear from the sandpile. In order for the model to be abelian it is essential that the parameters \(\cna,\cea,\csa,\cwa\) are non-negative. 
Explicitly if after \(n\) steps the sandpile is given by \(s_n(x)\) and on the \(n+1\)-st step the site \((x_1,x_2)\) topples, then the new sandpile is given by 
\begin{align*}
s_{n+1}(x_1,x_2) &= s_n (x_1,x_2) - cd\\
s_{n+1}(x_1+1,x_2)&= s_n(x_1+1,x_2) + \cea \\
s_{n+1}(x_1-1,x_2)&= s_n(x_1-1,x_2) + \cwa \\
s_{n+1}(x_1,x_2+1)&= s_n(x_1,x_2+1) + \cna \\
s_{n+1}(x_1,x_2-1)&= s_n(x_1,x_2-1) + \csa.
\end{align*}
The heights at all other sites remain unchanged.

We call the case with \(\cna=\csa=\cea=\cwa=1\) and $d>1$ the \emph{uniform Leaky-ASM}. When \(d=5/4\) a site can topple if it has at least \(5\) chips. Each time a site topples it sends one chip to each nearest-neighbor and leaks one chip. The standard ASM corresponds to the uniform Leaky-ASM with \(d=1\).

Our first main result is a limit shape theorem for the uniform Leaky-ASM with initial configuration \(s(x)=n\delta_{(0,0)}(x)\). We topple all sites until reaching the stable sandpile. A site is \emph{visited} if it topples during the stabilization process. Let \(D_{n,d}\) be the set of visited sites for the uniform Leaky-ASM with leakiness parameter \(d\) and initial configuration \(n\delta_{(0,0)}(x)\).

\begin{thm}
\label{thm:main}
Let \(d>1\) and \(r=\log n-\frac12\log\log n\). The boundary of the rescaled set of visited sites \(r^{-1}D_{n,d}\) of the uniform Leaky-ASM converges as $n\to\infty$ to the 
dual of the boundary of the gaseous phase in the amoeba of the Laurent polynomial 
\begin{equation}
\label{eq:P-uniform}
P(z,w)=\frac{4d- z -z^{-1}- w -w ^{-1}} {4(d-1)}.
\end{equation}
Moreover, if the limiting curve is scaled by $r$, then the boundary of \(D_{n,d}\) is within a constant distance from the scaled curve.
\end{thm}

\begin{remark}
More precisely, there exists a quantity $h(d)>0$ depending on $d$ only, such that in any direction, if $n$ is large enough, the unscaled visited region is within distance $h(d)$ of the limit curve scaled up by $r=\log n-\frac 12 \log\log n$. 
\end{remark}

\begin{remark}
\label{rem:limShapeParam}
The curve giving the limit shape can be parametrized as
\begin{align}
    \label{eq:MainThmLimitShape}
    -\left( \frac{1}{S(\wcr)}, \frac{a}{S(\wcr)}  \right) \text{ for } 0\leq a \leq 1,
\end{align}
and its reflections with respect to the coordinate axes and the diagonal $y=x$, where $S(w) = S(z_+(w),w)$ with $S(z,w)=-\log z-a\log w$, $z_+(w)$ is a portion of the spectral curve $P(z(w),w)=0$ given by \eqref{eq:zpm} and $\wcr=w_+(a,d)$ is a critical point of $S(w)$ given by \eqref{eq:wcr}.
\end{remark}

The limit shape in Theorem \ref{thm:main} simplifies when the leakiness converges to the extreme values of \(0\) and \(\infty\). When $d\to\infty$ the limit shape converges to the $L_1$ unit disk, while when $d\to 0$ it converges to the unit circle. The limit curves for various values of $d$ are shown in Figure \ref{fig:4dplots} and simulations of the Leaky-ASM with the same parameters are shown in Figure \ref{fig:4dsimulations}.

\begin{figure}
    \centering
        \includegraphics[width=14cm]{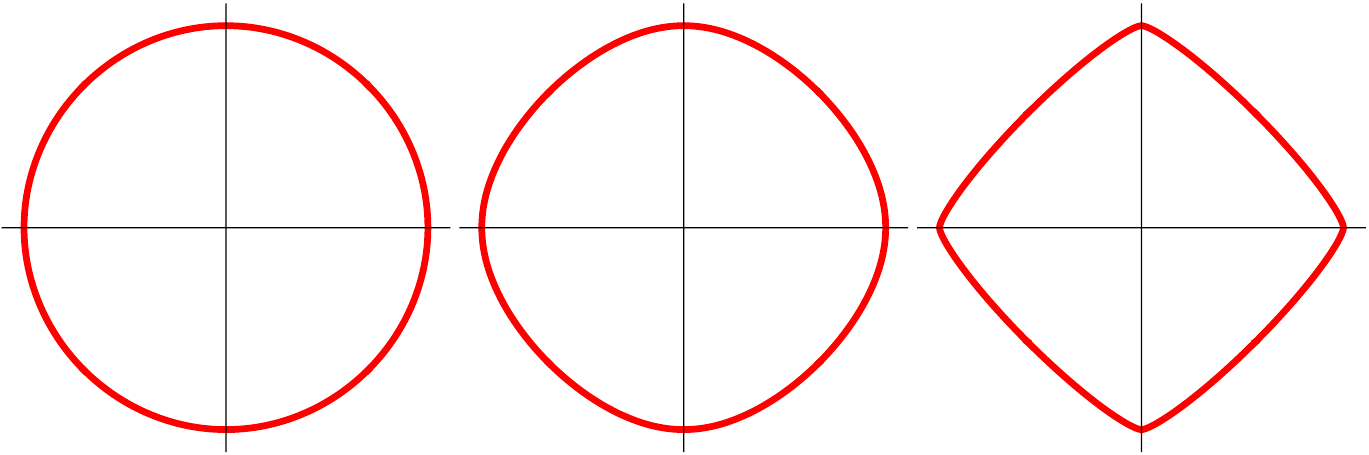}
        \caption{Limit shapes of the uniform Leaky-ASM with parameters, from left to right, $d=1.05,d=2$ and $d=100$. Each limit shape is equal (up to rescaling) to the dual of the boundary of the gaseous phase in the amoeba of \eqref{eq:P-uniform}.}
        \label{fig:4dplots}
\end{figure}

\begin{figure}
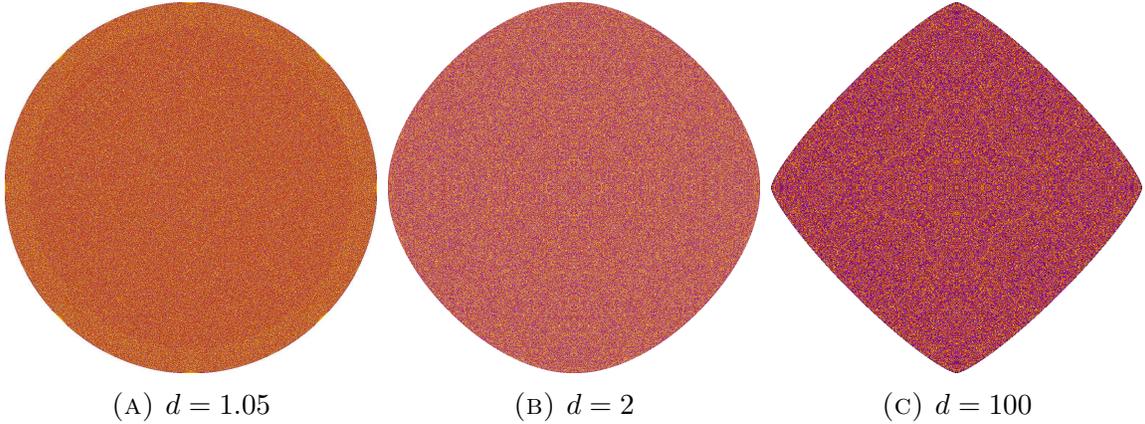

    \centering
    \begin{subfigure}[b]{0.3\textwidth}
        \includegraphics[width=\textwidth]{c10p100d105.pdf}
        \caption{\(d=1.05\)}
        \label{fig:dsmall}
    \end{subfigure}
    \begin{subfigure}[b]{0.3\textwidth}
        \includegraphics[width=\textwidth]{c10p500d2.pdf}
        \caption{\(d=2\)}
    \end{subfigure}
    \begin{subfigure}[b]{0.3\textwidth}
        \includegraphics[width=\textwidth]{c10p1000d100.pdf}
        \caption{\(d=100\)}
        \label{fig:dlarge}
    \end{subfigure}
    \caption{Simulations of the uniform Leaky-ASM with parameters, from left to right, $d=1.05$, $d=2$ and $d=100$}
    \label{fig:4dsimulations}
\end{figure}

The computations greatly simplify in the case of the north-east oriented Leaky-ASM, i.e. when \(\cna = \cea =1\), \(\csa=\cwa=0\) and $d>1$. We outline a simple argument giving the limit shape explicitly in this case.

As the leakiness \(d-1\) converges to \(0\), each topple in the Leaky-ASM converges to a topple of the standard ASM. 
In Theorem \ref{thm:leaky-to-ASM} we show that the stable configuration of the Leaky-ASM started with \(n\) chips at the origin and leakiness parameter \(d_n-1=t_n\), depending on \(n\), converges to the stable configuration of the standard ASM started with \(n\) chips at the origin and a background height of \(-1\) provided that \(t_n\) converges to \(0\) sufficiently fast as a function of \(n\). Figure \ref{fig:convergence} shows simulations of the stable configuration of the Leaky-ASM as \(t_n\) converges to \(0\). Figure \ref{fig:ASMb-1} shows the stable configuration of the standard ASM with background height \(-1\).

\begin{figure}
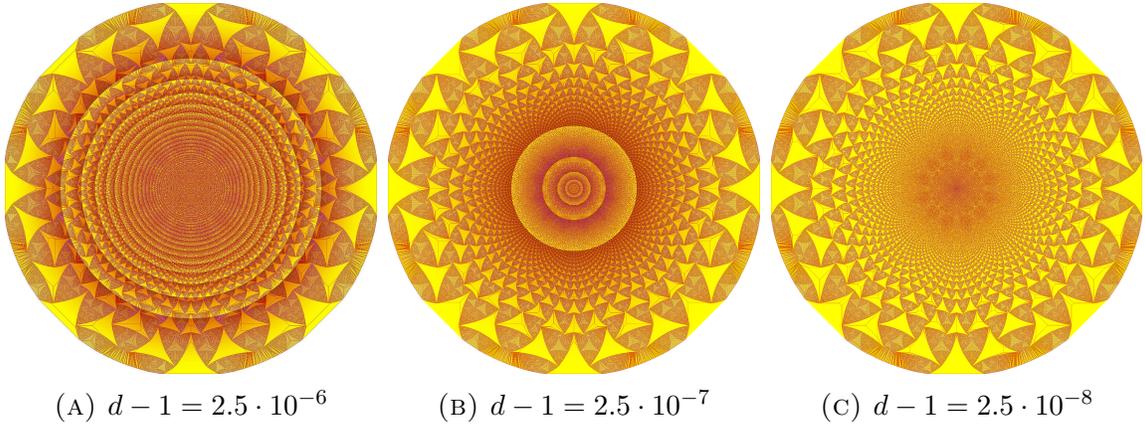

    \centering
    \begin{subfigure}[b]{0.3\textwidth}
        \includegraphics[width=\textwidth]{c10p7dm5.pdf}
        \caption{\(d-1=2.5 \cdot 10^{-6}\)}
        \label{fig:conv-5}
    \end{subfigure}
    \begin{subfigure}[b]{0.3\textwidth}
        \includegraphics[width=\textwidth]{c10p7dm6.pdf}
        \caption{\(d-1= 2.5 \cdot 10^{-7}\)}
        \label{fig:conv-6}
    \end{subfigure}
    \begin{subfigure}[b]{0.3\textwidth}
        \includegraphics[width=\textwidth]{c10p7dm7.pdf}
        \caption{\(d-1= 2.5 \cdot 10^{-8}\)}
        \label{fig:conv-7}
    \end{subfigure}
    \caption{Convergence of the Leaky-ASM to the ASM with background height $-1$ and \(n=10^7\) chips}\label{fig:convergence}
\end{figure}

As a consequence of Theorem \ref{thm:leaky-to-ASM}, it is of interest to study the limit shape of the Leaky-ASM when $t_n\to 0$ as $n\to\infty$. We prove:

\begin{thm}
\label{thm:leakTo0}
Consider the uniform Leaky-ASM with parameter $d_n=1+t_n$ started with $n$ chips at the origin. If $t_n\to 0$ as $n\to\infty$, we have the following:
\begin{itemize}
\item If $t_n\to 0$ slower than any power of $n$, then the boundary of the region \(D_{n,d}\) scaled down by $\log(n)/\sqrt{t_n}$ converges to a circle.
\item If $t_n\to 0$ as $1/n^{1-\alpha}$ with $0<\alpha<1$, then the region \(D_{n,d}\) scaled down by $\log(n)/\sqrt{t_n}$ is bounded between two circles, the ratio of whose radii converges to $\alpha$.
\end{itemize}
\end{thm}

\begin{figure}
    \centering
        \includegraphics[width=9cm]{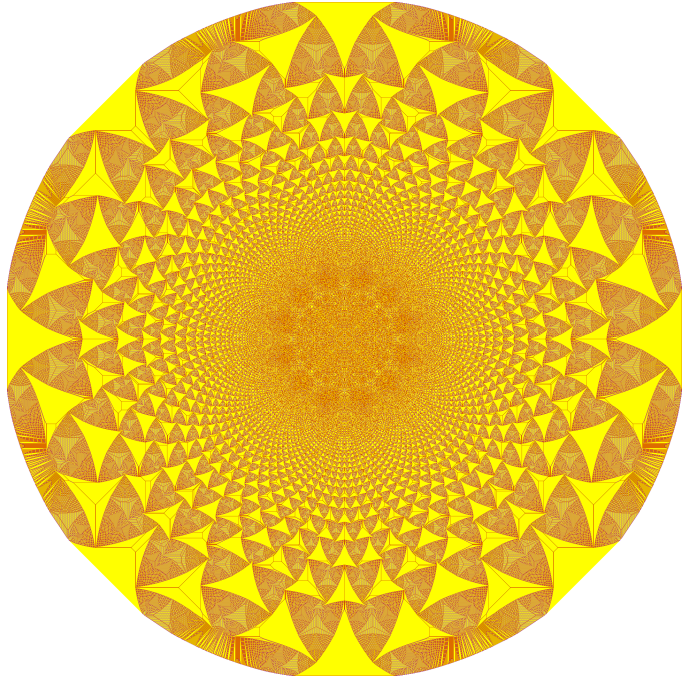}
        \caption{Standard ASM with background height $-1$ and \(10^7\) chips}
        \label{fig:ASMb-1}
\end{figure}

We obtain all of our limit shape results by relating the Leaky-ASM to the \emph{killed random walk} (KRW). In the KRW a random walker starts at the origin and at each step either gets killed with probability \(1-\frac{1}{d}\) or moves one step in either of the north, east, south, or west directions with respective probabilities $ \frac{\cna}{cd}, \frac{\cea}{cd}, \frac{\csa}{cd} $, and $\frac{\cwa}{c d}$. Let $P_d(x)$ be the probability that the walker dies at site $x$. We show:

\begin{prop}\label{prop:shape-vs-death}
    For any $x\in\mathbb{Z}^2$ we have:
\begin{enumerate}
    \item If \(P_d(x)< \dfrac{c(d-1)}{n}\) then \(x \not \in D_{n,d}\).
    \item if \(P_d(x)\geq \dfrac{c d}{n} \) then \(x\in D_{n,d}\).
\end{enumerate}
\end{prop}

\begin{remark}
 Unfortunately Theorem \ref{thm:leakTo0} does not give useful bounds for the region visited by the Leaky-ASM when $d_n\to 1$ quickly because the two level curves 
 \[P_{d_n}(x)=\dfrac{c(d_n-1)}{n} \qquad \text{and} \qquad  P_{d_n}(x)=\dfrac{c d_n}{n}\]
 of $P_{d_n}$ are too far apart in that regime.
\end{remark}

\subsection{Relation of the Leaky-ASM to the ASM with Background Height -1}
%
\begin{thm}
\label{thm:leaky-to-ASM}
As \(d\to 1\) the Leaky-ASM model converges pointwise to the ASM where each site has background height $-1$.
\end{thm}

\begin{remark}
We can formulate this more precisely as follows. Given an initial sandpile configuration \(h(x)\), let \(s^h_{\infty,d}\) denote the stable sandpile configuration with leakiness \(d\). The configuration \(s^h_{\infty,1}(x)\) corresponds to the standard ASM without dissipation. The statement of the theorem is equivalent to 
\begin{align*}
    \lim_{d\to 1} s^{n\delta_{(0,0)}}_{\infty,d}(x)=s^{n\delta_{(0,0)}-1}_{\infty,1}(x) \text{ for all } x \in \Z^2.
\end{align*}

\end{remark}

\begin{proof}
We establish this by coupling two sandpiles. Consider the following slightly modified version of the ASM: a site topples if it has at least $5$ chips. When it topples, it sends $1$ chip to each neighbor. Start this model with $n$ chips at the origin and no chips anywhere else and let $m_n$ be the maximum number of times any site topples until no sites can topple. Let $B_k:\mathbb{Z}^2 \to \Z_{\geq 0}$ be the height of the configuration after $k$ firings. 

On the other hand, consider the Leaky-ASM started with a stack of $n$ chips at the origin and leakiness $d_n-1=t_n$ depending on $n$. Suppose $t_n<\frac{1}{4m_n}$. Let $L_k:\mathbb{Z}^2\to\mathbb{R}_{\geq 0}$ be the height of the configuration after $k$ firings. 

Of course the functions $B_k$ and $L_k$ depend on the order in which the sites topple. We use a simple induction argument to show that we can couple the two models by always firing the same sites. Suppose the first $k$ firings have been done at identical locations. Fix any site $x$. Since every neighbor of $x$ has toppled the same amount of times in both models and in both models each time a neighbor toppled, $x$ received $1$ chip, the number of chips $x$ received from all neighbors up to time $k$ would be the same. On the other hand every time $x$ toppled in the leaky model it lost $4d_n=4(1+t_n)$ chips and in the non-leaky case $4$ chips. Since $x$ topples at most $m_n$ times in the non-leaky version, we have
\begin{align*}
B_k(x)-1<B_k(x)-4t_nm_n\leq L_k(x)<B_k(x).
\end{align*}
This implies that, since the threshold for firing the Leaky-ASM is $4(1+t_n)$ and for the modified ASM is $5$, site $x$ can topple in either both or neither of the models.

It follows that both models will reach their final configuration at the same time $k_f$ and we will have $B_{k_f}(x)=\lceil L_{k_f}(x)\rceil$ for every site $x$. 

Finally, notice that in the modified ASM whenever a site topples, it leaves at least $1$ chip behind, so $1\leq B_{k_f}(x)<5$. Subtracting $1$ from each site is equivalent to starting the model with height $n-1$ at the origin and a well of depth $1$ everywhere else and running the standard ASM.

\end{proof}

\subsection{Outline} The paper is organized as follows. In Section \ref{sec:KilledWalk} we introduce the killed random walk (KRW) and relate death probabilities of the walk to coefficients in a Laurent expansion. In Section \ref{sec:WalkVsSandpiles} we show the connection between the Leaky-ASM and the KRW by connecting the values of the odometer function of the Leaky-ASM to the death probabilities of the KRW.
In the short Section \ref{sec:2dShape} we illustrate simple calculations giving the limit shape in the case of the north-east oriented Leaky-ASM. Section \ref{sec:4dshape} contains the proof of the main theorem. In the proof we express the death probabilities of the KRW as a contour integral and compute its asymptotics using the steepest descent method. This gives us the explicit parametrization of the limit shape given in Remark \ref{rem:limShapeParam}. In Section \ref{sec:amoebae} we discuss the connection of the limit shape to the dual of the amoeba of the associated spectral curve. Section \ref{sec:LeakTo0} contains the steepest descent analysis in the proof of the second main result.

\section{The Killed Random Walk}
\label{sec:KilledWalk}
In this section we relate the death probabilities of the killed random walk to the coefficients in the Laurent series expansion of a rational function. Let $X_1,X_2,\dots$ be a sequence of independent identically distributed random variables with common distribution
\begin{align*}
\begin{gathered}
        P\{X_j = (1,0)\}  = \frac{\cea}{cd}, \qquad P\{X_j = (-1,0)\}  = \frac{\cwa}{cd}, \\
        P\{X_j = (0,1)\}  = \frac{\cna}{cd}, \qquad P\{X_j = (0,-1)\}  =\frac{\csa}{c d}, \\
        P\{X_j = (0,0)\}  = 1-\frac{c}{cd}=1-\frac 1d.
\end{gathered}
\end{align*}
For notation of transition probabilities we write 
\[\ptran{x}{y} = P(X_j =y-x).\]

We will interpret $X_j=(0,0)$ to mean that the walk is killed, so we define $$K_n:=\prod_{i=1}^n 1_{X_i\neq (0,0)}$$ to be the indicator whether the walk has been killed by time $n$ or not. $K_n=1$ means that the walker is still alive after the $n$-th step. Formally, by the \emph{killed random walk} (KRW) started at \(x\in \Z^2\) we mean the sequence $S_n$ of random variables defined by
\begin{equation*}
S_n =  x + K_1 X_1 + \cdots + K_n X_{n}.
\end{equation*}
For the KRW started at the origin let 
\[P_d(x) = \text{ Probability KRW dies at } x=P(S_{\min\{i:K_i=0\}}=x).\]
We assume \(d>1\) and define the Laurent polynomial
\begin{equation}
\label{eq:P}
P(z,w) = \frac{cd- \left( \cna z + \csa z^{-1}+ \cea w + \cwa w ^{-1}\right) } {c(d-1)}.
\end{equation}
We connect the death probabilities to coefficients of monomials in a Laurent expansion of $P^{-1}(z,w)$. 

\begin{lem}
\label{lem:PdCoeffP}
We have
\[[ P^{-1} (z,w) ]_{ij}=P_d(i,j),\]
where the left-hand side is the coefficient of $z^iw^j$ in the Laurent series expansion of $P^{-1}$ in the region 
\begin{equation}
\label{eq:expansion-region}
  \frac{\cna |z| + \csa |z^{-1}|+\cea |w| + \cwa |w^{-1}|}{cd} <1.
\end{equation}

\end{lem}

\begin{proof}

First, we expand \(P^{-1}(z,w)\) as a Laurent series 
\begin{align*}
P^{-1}(z,w) & = \frac{c(d-1)}{ cd-\left( \cna z + \csa z^{-1}+ \cea w + \cwa w ^{-1}\right)} \\
& =\frac{c(d-1)}{c d} \frac{1}{ 1-\left( \cna z + \csa z^{-1}+ \cea w + \cwa  w^{-1}\right)/(cd)} \\
& = \frac{d-1}{d}\sum_{k=0}^{\infty}  (cd)^{-k} \left( \cna z + \csa z^{-1}+ \cea w + \cwa w^{-1} \right) ^k
\end{align*} 
which converges in the region \eqref{eq:expansion-region}.

A path which dies at site \((i,j)\) must pass through the site then die at the next step. For a site \((i,j)\in \Z^2\) let \(\Gamma_k(i,j)\) be the set of paths from \((0,0)\) to \((i,j)\) with length \(k\). For a path \(\gamma\in \Gamma_k\) suppose that it takes \(n_{\cna}\) steps up, \(n_{\csa}\) steps down, \(n_{\cea}\) steps right and \(n_{\cwa}\) steps left. Let the weight of a path \(w(\gamma)\) be the product of the weights along the path, i.e., \(w(\gamma) = \cna^{n_{\cna}} \csa^{n_{\csa}} \cea^{n_{\cea}} \cwa^{n_{\cwa}}\). Since
\begin{align*}
\left( \cna z + \csa z^{-1}+ \cea w + \cwa w^{-1} \right) ^k & = \sum_{(i,j)\in \Z^2} \sum_{\gamma_k \in \Gamma_k(i,j)} w(\gamma_k) z^i w^j ,
\end{align*}
we get
\begin{align*}
P^{-1}(z,w)
& =  \frac{d-1}{d}\sum_{k=0}^{\infty} (cd)^{-k}  \sum_{(i,j)\in \Z^2} \sum_{\gamma_k \in \Gamma_k(i,j)} w(\gamma_k) z^i w^j \\
& =  \sum_{(i,j)\in \Z^2} \left( \sum_{k=0}^{\infty} \frac{d-1}{d} (cd)^{-k} \sum_{\gamma_k \in \Gamma_k(i,j)} w(\gamma_k)\right) z^i w^j,
\end{align*}
where we could exchange the order of summation since by \eqref{eq:expansion-region} the series above converges absolutely.

The coefficient \( \displaystyle\frac{d-1}{d} (cd)^{-k} \sum_{\gamma_k \in \Gamma_k(i,j)} w(\gamma_k)\) is the probability that the walker dies at site \((i,j)\) on the \((k+1)\)-st step. Summing over all \(k\) we find that the coefficient of \(z^i w^j\) is the probability that the walker dies at site \((i,j)\).
\end{proof}

\section{Connection Between The Killed Random Walk and Leaky Sandpiles}
\label{sec:WalkVsSandpiles}
In this section we relate the death probabilities of the killed random walk to the region visited by the Leaky-ASM when started with a single large stack of chips at the origin. 

We focus on the case in which \(n\) chips start at the origin \((0,0)\). This initial configuration corresponds to the point mass \(n \delta_{(0,0)}(x)\). Let \(D_n\) be the set of sites which are visited with this initial configuration. 

A very useful tool to study sandpile models is the \emph{odometer function} introduced in \cite{d2006}. It is defined as
\[u(x): = \text{total mass emitted from \(x\).}\]
Note, that in our interpretation $u(x)$ includes both the mass sent to the neighbors and the mass leaked.

Define the operator $T$ by
\begin{align}
\label{eq:T}
T u(x): & =  \left( \sum_{y \sim x} \frac{\ctran{y}{x}}{cd}u(y) \right)- u(x) \\
\nonumber
& = \text{mass received by } x -  \text{mass emitted by } x,
\end{align}
where $y\sim x$ means the site $y$ is a neighbor of the site $x$ and $\ctran{y}{x}$ stands for the number of chips $x$ receives when $y$ topples (e.g. if $x$ is to the east of $y$, then $\ctran{y}{x}=\cea$), so 
$\frac{\ctran{y}{x}}{cd}$ is the portion of the chips emitted from $y$ that are sent to $x$. 

\begin{remark}
Note, that we can express the operator $T$ in terms of the weighted Laplacian \(\Delta u(x) = \sum_{y \sim x} e_{y\to x} u(y) - u(x)\), where $e_{y\to x}$ would be the weight of the edge between $x$ and $y$ with $\sum_{x:y\sim x}e_{y\to x}=1$. We have 
\begin{equation*}
T = \frac{1}{d} \Delta  -\left(\frac{d-1}{d}  \right) I.
\end{equation*}
In the case of the ASM we have $d=1$ so the identity term disappears and we get $T=\Delta$ with $\Delta$ being the standard Laplacian since all the edge weights are $1/4$ in the ASM. 
\end{remark}

We start the Leaky-ASM with \(n\) chips at the origin and run it until it stabilizes to a final configuration \(f(x)\). The odometer satisfies the equation
\begin{align}
\label{eq:Tu}
T u(x) = f(x) - n \delta_{(0,0)}(x).
\end{align}

Note, however, that $\frac{\ctran{y}{x}}{cd}=\ptran{y}{x}$, which gives a probabilistic interpretation to the operator $T$ given in the next lemma.

\begin{lem}\label{lem:Tinverse}
Applying the operator $T$ to the death probabilities $P_d$ we obtain
\begin{align}
\label{eq:TPd}
 T P_d(x) = - \frac{d-1}{d} \delta_{(0,0)} (x).
 \end{align}
\end{lem}

\begin{proof}
Suppose the position of the walker after \(k\) steps is \(S_k\).  
Let \(P_d^k(x)\) be the probability that the walker dies after \(k\) steps at site \(x\). We have \(P_d(x)= \sum_{k=1}^{\infty} P_d^k(x)\). Note that a step consists of either moving to a neighboring cell or dying.

Assume \(x\neq (0,0)\). If \(x\neq (0,0)\) then the walker cannot die at $x$ in fewer than \(2\) steps since it will take at least $1$ step to get to $x$ and another to die there, and so we have
\begin{align*}
P_d(x) &=\sum_{k=2}^{\infty} P(S_{k-1}= x,K_{k-1}=1) \frac{d-1}{d}\\
& = \sum_{k=2}^{\infty} \sum_{y \sim x} P(S_{k-2}=y,K_{k-2}=1) P(S_{k-1}=x | S_{k-2}=y,K_{k-2}=1) \frac{d-1}{d} \\
& =  \sum_{k=2}^{\infty} \sum_{y \sim x} P(S_{k-2}=y,K_{k-2}=1) \ptran{y}{x} \frac{d-1}{d} \\
& =  \sum_{k=2}^{\infty} \sum_{y \sim x} P_d^{k-2}(y)  \ptran{y}{x}   =  \sum_{y \sim x} \ptran{y}{x} \sum_{k=2}^{\infty} P_d^{k-2}(y)     =  \sum_{y \sim x} \ptran{y}{x}P_d(y)   \\
&= T P_d(x) + P_d(x).
\end{align*}
It follows that \(T P_d(x)=0\) for \(x\neq (0,0)\).

If \(x=(0,0)\) then the walker can die in 1 step with probability \(\frac{d-1}{d}\) and so using the above computation we have 
\begin{align*}
P_d((0,0)) &=\frac{d-1}{d} + \sum_{k=2}^{\infty} P(S_{k-1}= x,K_{k-1}=1) \frac{d-1}{d} \\
& = \frac{d-1}{d}+ T P_d(0,0) + P_d(0,0),
\end{align*}
which implies \(TP_d(0,0) = -\frac{d-1}{d}\).
\end{proof}

We can use the fact that the odometer and the death probabilities satisfy similar equations to relate them to each other.

\begin{lem}\label{lem:prob-limit-shape}
For any $x\in\mathbb{Z}$ we have:
\begin{enumerate}
\item If \(P_d(x)< \dfrac{c(d-1)}{n}\), then \(u(x)=0\).
\item If \(P_d(x)\geq \dfrac{c d}{n} \), then \(u(x)\geq cd\).
\end{enumerate}
\end{lem}

\begin{proof}
    Combining \eqref{eq:Tu} and \eqref{eq:TPd} and using the fact that \(T: L^1(V)\to L^1(V)\) is a linear operator we have 
    \begin{align*}
    T \left( \frac{d-1}{dn } u(x) - P_d(x) \right)& = \frac{d-1}{dn } f(x) .
    \end{align*}
    Since \(f(x)\) is the stabilized configuration, no site can topple, so we have \(0\leq f(x)< cd \), which implies
    \begin{align}
    \label{eq:T-Pd}
    0 \leq T \left( \frac{d-1}{dn } u(x) - P_d(x) \right) < \frac{c(d-1)}{n}.
    \end{align}
    Lemma \ref{lem:Tinverse} shows that \(T\) is invertible with inverse
    \begin{align*}
        (T^{-1}f)(y) = -\frac{d}{d-1} \sum_{x\in V} P_d(y-x) f(x).
    \end{align*}
    Moreover, the constant function $1$ is an eigenvector of $T$ with eigenvalue $-\frac{d-1}{d}$. Thus it is also an eigenvector of $T^{-1}$ with eigenvalue $-\frac{d}{d-1}$. Applying $T^{-1}$ to \eqref{eq:T-Pd} and using this fact along with the fact that \(T^{-1}\) has all negative coefficients, we obtain
\begin{align*}
0 &\geq\frac{d-1}{dn } u(x) - P_d(x) > - \frac{c(d-1)}{n} \frac{d}{d-1} 
\end{align*} 
which gives
\begin{align*}
    \frac{dn}{d-1} \left( P_d(x)- \frac{cd}{n}\right)&< u(x) \leq \frac{dn}{d-1} P_d(x).   
\end{align*}

If \(P_d(x) \geq \frac{cd}{n}\) then \(u(x)>0\). On the other hand, every time a site topples, it emits mass $cd$, so if a site has emitted mass then it must have emitted mass at least \(cd\) so in fact \(u(x) \geq cd\). If \(\frac{dn}{d-1} P_d(x)<cd\) or, equivalently, \(P_d(x)< \frac{c(d-1)}{n}\) then \(u(x)<cd\) which implies that $x$ has not toppled at all and so \(u(x)=0\).

\end{proof}

\begin{remark}
 Proposition \ref{prop:shape-vs-death} follows immediately from Lemma \ref{lem:prob-limit-shape}
\end{remark}

\section{The Limit Shape of the 2-Directional Leaky-ASM}
\label{sec:2dShape}
In this section we show how just using Stirling's approximation the limit shape can be computed in the much simpler case of the 2-directional Leaky-ASM, where whenever a site topples, it sends equal number of chips only to its northern and eastern neighbors and leaks some. In our notation this corresponds to the case $\cna=\cea=1$, $\csa=\cwa=0$ and $d>1$.

In this case we have 
\begin{align*}
P^{-1}(z,w) & = \frac{d-1}{d} \sum_{k=0}^{\infty} (2d)^{-k} (z+w)^k,
\end{align*}
which implies
\begin{align*}
[P^{-1}(z,w)]_{i j} = \frac{d-1}{d} \binom{i+j}{i} (2d)^{-(i+j)}.
\end{align*}
Let \(i=r\) and \(j=r a\) with $a\geq 0$. By symmetry, it is enough to compute the shape of the visited region in the region $0<x,0<y<x$, so we can assume \(0\leq a \leq 1\). By Lemma \ref{lem:PdCoeffP} we have
\begin{align*}
P_d(r,ar) &= \frac{d-1}{d} \binom{r(1+a)}{r} (2d)^{-r(1+a)}.
\end{align*}
Applying Stirling's approximation, \(m! \sim \sqrt{2\pi m} (m/e)^m\) with \(m=(1+a)r\) gives that when $r$ is large, we have
\begin{align*}
P_d(r,ar)  &\sim 
\frac{d-1}{d} \frac{1}{\sqrt{2\pi r}}\sqrt{\frac{1+a}{a}}\left( \frac{1 }{a^a }\left( \frac{1+a}{2d} \right)^{a+1} \right)^{r}.
\end{align*}
By Lemma \ref{lem:prob-limit-shape} to compute the limit shape we need to find $r_1,r_2$ such that $$P_d(r_1,ar_1) \sim \frac{cd}{n}\text{ and }P_d(r_2,ar_2) \sim \frac{c(d-1)}{n}$$ as \(n\to \infty\). We get $r_1,r_2=-\frac{\log{n}}{g}(1+o(1))$ with $r_2-r_1=-\frac{\log\frac{d-1}d}{g}+o(1)$, where 
\[g(a) =\log \left( \frac{1 }{a^a } \left(\frac{1+a}{2d} \right)^{1+a} \right).\]

Thus, if we scale down by $\log(n)$ the region visited by the 2-directional Leaky-ASM started with $n$ chips at the origin, we get the parametric curve consisting of \((x,y) = (-1/g(a),-a/g(a))\) for \(0\leq a \leq 1\) and its reflection about the line \(y=x\).

Graphs of the limit shape curves for $d=1.05,d=2$ and $d=10^5$ are shown in Figure \ref{fig:2dplots}. The shapes in the plots are scaled down not by $\log(n)$ but by $\log(n)/\log(d)$ so that they are all of the same scale. Simulations of the sandpile with those same parameters are shown in Figure \ref{fig:2dsimulations}. 

The limit shape is in fact a portion of the dual of the amoeba of the corresponding spectral curve. See Section \ref{sec:amoebae} for details.

\begin{remark}
\label{rem:limitTriangle}
Note, that the limit shape converges to a triangle when $d\to\infty$. A heuristic explanation of this is that when $d$ is large, every time a site topples, lots of chips leak, so it is much harder to get an extra step further, so the furthest sites reached all have the same $L_1$ distance from the origin. 
\end{remark}

\begin{figure}
    \centering
        \includegraphics[width=12cm]{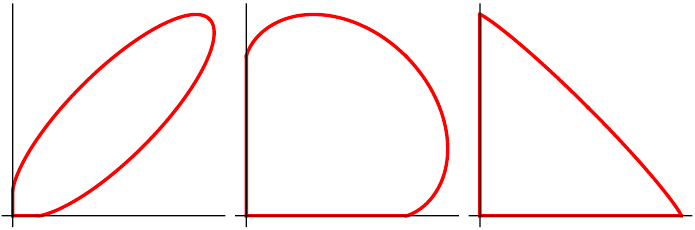}
        \caption{Limit shapes of the 2-directional Leaky-ASM with parameters, from left to right, $d=1.05,d=2$ and $d=10^5$. Each limit shape is (up to rescaling) a certain portion of the dual of the amoeba of \(P(z,w)=\frac{2d-z-w}{2(d-1)}\).}
        \label{fig:2dplots}
\end{figure}

\begin{figure}
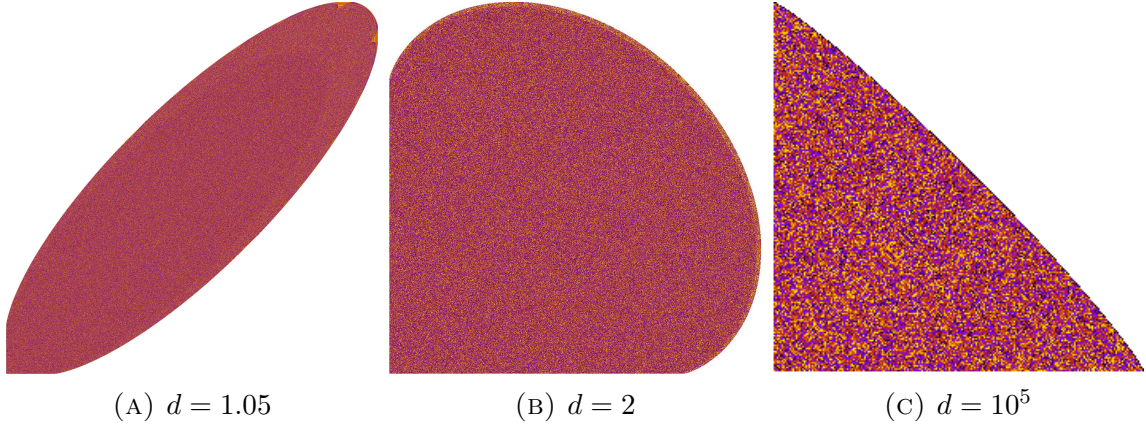

    \centering
    \begin{subfigure}[b]{0.3\textwidth}
        \includegraphics[width=\textwidth]{c10p100d105ne}
        \caption{\(d=1.05\)}
    \end{subfigure}
    \begin{subfigure}[b]{0.3\textwidth}
        \includegraphics[width=\textwidth]{c10p500d2ne}
        \caption{\(d=2\)}
    \end{subfigure}
    \begin{subfigure}[b]{0.3\textwidth}
        \includegraphics[width=\textwidth]{c10p1000d10p5ne.pdf}
        \caption{\(d=10^5\)}
    \end{subfigure}
    \caption{Simulations of the 2-directional Leaky-ASM with parameters, from left to right, \(d=1.05\), \(d=2\) and \(d=100\).}
    \label{fig:2dsimulations}
\end{figure}

\section{The Limit Shape of the Leaky-ASM}
\label{sec:4dshape}
\subsection{A Contour Integral Representation of the Death Probabilities}
In this section we obtain the limit shape of the Leaky-ASM with parameter $d>1$. We have $\cna=\cea=\csa=\cwa=1$.

By Lemma \ref{lem:prob-limit-shape} we should determine the points $x\in\mathbb{Z}^2$ such that $P_d(x)$ is of order $\frac{1}{n}$. When $n\to\infty$, the $x$'s for which $P_d(x)$ if order $\frac{1}{n}$ will have $\|x\|_2\to\infty$. By Lemma \ref{lem:PdCoeffP}, instead, we can compute coefficients in the Laurent expansion of $P^{-1}(z,w)$, which we will do by writing the coefficients as contour integrals and using the method of steepest descent. 

First, recall, that we assume \eqref{eq:expansion-region} which in this case becomes
\[ |z| +  |z^{-1}|+ |w| +  |w^{-1}| <4d.\]
The equation \(z P(z,w)=0\) is quadratic in \(z\) and therefore we can write 
\begin{align*}
P(z,w) =- \frac{(z- z_+(w))(z-z_-(w))}{c(d-1)z}
\end{align*}
where the roots satisfy \(z_+ z_-=1\), \(z_++z_- = 4d- w-1/w\), and are ordered with \(|z_-|<1<|z_+|\). Explicitly we have 
\begin{align}
\label{eq:zpm}
z_{\pm}(w) & = \frac{(4d-w-1/w) \pm\sqrt{(4d-w-1/w)^2-4} }{2}.
\end{align}
Note, that, to simplify expressions, we will often write $z_\pm$ for $z_\pm(w)$. In cases where this will cause ambiguities we will give the argument explicitly.

\begin{lem}
If \(|w|=1\) then \(z_-,z_+\in \R_{+}\). In particular \(0<z_-< 1< z_+\).
\end{lem}

\begin{proof}
If $|w|=1$, then $\frac 1w=\bar{w}$, where the bar stands for complex conjugation. Thus $w+\frac 1w=2\Re(w)$. Since $|w|=1$, we have $\Re(w)\leq 1$ so $$4d-w-\frac 1w\geq 4d-2\Re(w)\geq 4d-2>2.$$ 
It follows that the expression under the square root in \eqref{eq:zpm} is positive so $z_\pm$ are real and positive.
\end{proof}

First, we extract the coefficient of $z^j$ for $j\in\mathbb{Z}$.
\begin{lem}
\label{lem:coeffz}
For $j\in\mathbb{Z}$ the coefficient of \(z^j\) is given by 
\begin{equation*}
[z^j]P^{-1}(z,w) =\displaystyle \frac{c(d-1)z_+^{-|j|}}{ z_+-z_+^{-1}}.
\end{equation*}
\end{lem}

\begin{proof}
First, split $P^{-1}$ as  
\begin{align*}
P^{-1}(z,w) & = \frac{-z c(d-1)}{ (z-z_-)(z-z_+)} = \frac{-z c(d-1)}{ z_--z_+} \left( \frac{1}{ z- z_-} - \frac{1}{z-z_+} \right).
\end{align*}
We are interested in the expansion in a neighborhood of \((z,w)=(1,1)\) so we may assume \(z_-< |z|< z_+\). We have
\begin{align*}
P^{-1}(z,w) & = \frac{-zc(d-1)}{ z_- - z_+} \left( \frac{1}{ z-z_-} - \frac{1}{z-z_+} \right)\\
& = \frac{-zc(d-1)}{ z_- -z_+} \left( \frac{1}{z} \sum_{k=0}^{\infty} \left ( \frac{z_-}{z} \right)^k + \frac{1}{z_+}\sum_{k=0}^{\infty} \left ( \frac{z}{z_+} \right)^k \right)\\
& = \frac{c(d-1)}{z_+ - z_+^{-1}}\left( \sum_{k=0}^{\infty} \left ( \frac{1}{z z_+} \right)^k   +\sum_{k=1}^{\infty} \left ( \frac{z}{z_+} \right)^k \right)\\
& = \frac{c(d-1)}{z_+ - z_+^{-1}} \sum_{k=-\infty}^{\infty}  z^k z_+^{-|k|}
,
\end{align*}
where we used that \(z_-=1/z_+\).
\end{proof}

\begin{remark}\label{rmk:line-prob}
This gives a rather simple formula for the probability of dying on the vertical line \(L_j:=\{(j,k)\in \Z^2 : k \in \Z \})\), since 
\begin{align*}
\sum_{k\in \Z} P_d(j,k) = [z^j]P^{-1}(z,1) = \frac{c(d-1) z_+^{-|j|}(1)}{ (z_+(1)-z_+^{-1}(1))} = \frac{c(d-1) z_+(1)^{1- |j|}}{z_+^2(1)-1},
\end{align*}
where \(z_+(1)=2d -1 + \sqrt{(2d-1)^2-1}\). 
\end{remark}

Since we have $\cna=\cea=\csa=\cwa=1$, the region reached by the Leaky-ASM will have $8$-fold symmetry. It will be symmetric with respect to the two axes and the lines $y=\pm x$. Thus it is enough to understand the values of \(P_d(x)  \) when $x=(r,ar)$ with $r>0$ and \(a\in [0,1]\). We will need to find $r$ such that $P_d(r,ar)$ is of order $1/n$ as $n\to\infty$. Such an $r$ will need converge to infinity, so next we will compute the asymptotics of $P_d(r,ar)$ when $r\to\infty$.

By Lemmas \ref{lem:PdCoeffP} and \ref{lem:coeffz} we have 
\begin{align*}
P_d(r, ar)& = [w^{ra}]\displaystyle \frac{c(d-1) z_+^{-|r|}}{ z_+ - z_+^{-1}} \\
&  =  \frac{c(d-1)}{2 \pi i} \oint_{C} \frac{z_+^{-r}}{ z_+- z_+^{-1}} \frac{dw}{w^{ra+1}} \\
& = \frac{c(d-1)}{ 2\pi i} \oint _C \frac{1}{( z_+ - z_+^{-1}) w} e^{-r \log (z_+ w^a)} dw,
\end{align*}
where $C$ is the counter-clockwise contour $|w|=1$. 
From \eqref{eq:zpm} we have $$z_+ -1/z_+ =\sqrt{(4d-w-1/w)^2-4} $$ and so we can define
\begin{equation}
\label{eq:G}
G(w) =\frac{1}{ ( z_+- z_+^{-1}) w }= \frac{1}{w \sqrt{(4d-w-1/w)^2-4}}
\end{equation}

and 
\begin{align}
\label{eq:Sdeff}
S(w) &=- \log (z_+ w^a) 
\\\label{eq:Sinden}& = \log \left(\frac{2 w^{-a}}{4d-w-\frac{1}{w}+\sqrt{\left(4d-w-\frac{1}{w}\right)^2-4}}\right)\\
\label{eq:Sinnum}&=\log \left(\frac{4d-w-\frac{1}{w}-\sqrt{\left(4d-w-\frac{1}{w}\right)^2-4}}{2 w^a} \right).
\end{align}
In the newly defined notation we get
\begin{align}
\label{eq:PContInt}
P_d(r,ar) & = \frac{c(d-1)}{2 \pi i} \oint_C G(w) e^{ r S(w)} dw.
\end{align} 

\subsection{Asymptotics of the Death Probabilities}

\subsubsection{Critical points of the exponent}
To apply the steepest descent method we need to understand the critical points of $S(w)$. 

\begin{lem}\label{lem:real-roots}
The function \(S(w)\) has two real critical points when \(0<a\leq 1\) and \(d >1\). If $w_-$ is the smaller one and $w_+$ the larger one, then \[-1< w_-<0 <1<w_+\]
and $S(w_+)\in\mathbb{R}$.

\end{lem}

\begin{proof}
From \eqref{eq:Sdeff} we have
\begin{align}
\label{eq:Sp}
S'(w) & = - \left( \frac{z_+'}{z_+} + \frac{a}{w} \right),
\end{align}
which implies critical points occur when 
\begin{align}
\label{eq:Speqn}
a &=- \frac{w z_+'}{z_+} = \frac{w-\frac{1}{w}}{\sqrt{\left(4d-w-\frac{1}{w}\right)^2-4}}. 
\end{align}

Squaring this equation we get a degree $4$ polynomial equation for $w$ which can, of course, be solved explicitly, giving roots
\begin{align*}
\frac{-2a^2 d+ u_\pm\pm2 a \sqrt{d^2(1+a^2)- du_\pm}}{1-a^2},
\end{align*}
where 
\begin{equation}
\label{eq:u}
u_\pm=\pm\sqrt{4 a^2d^2 +(1-a^2)^2},
\end{equation}
which are necessarily real. We have 
\begin{equation}
\label{eq:ineqForU}
\left(d^2(1+a^2) \right)^2-d^2u_\pm^2=d^2(d^2-1)(1-a^2)^2>0,
\end{equation}
 so all four roots are real. Computing $w-1/w$ we get $\pm4 a \sqrt{d^2(1+a^2)- du_\pm}/(1-a^2)$ which, together with \eqref{eq:Speqn}, implies that only the two roots that have the plus sign in front of the square root in this expression are actual critical points of $S(w)$. They are
\begin{align}
\label{eq:wcr}
w_\pm& =\frac{-2a^2 d+ u_ \pm + 2a \sqrt{d^2(1+a^2)- du_ \pm}}{1-a^2}.
\end{align}

Differentiating $w_+$ with respect to $a$ and simplifying we obtain
\begin{align*}
\frac{dw_+}{da} 
&=\frac{4 ad (d(1+a^2)-u_+)+\sqrt{d \left(d(1+a^2)-u_+\right)} \left(2 u_+-4 a^2
   d\right)}{\left(1-a^2\right)^2 u_+}.
\end{align*}
From the definition of $u_+$ it is clear that $u_+\geq 2ad\geq 2a^2d$ when $0\leq a\leq 1$. On the other hand \eqref{eq:ineqForU} implies that $d(1+a)^2\geq u_+$. These two inequalities together imply that $w_+$ is an increasing function of $a$. Since $w_+=1$ when $a=0$ we get that $w_+>1$ for all $a>0$. An identical computation shows that \(\frac{d}{da} w_->0\). Since \(w_-=-1\) when \(a=0\), this implies \(-1<w_-<0\).

For the last claim, first note that 
$$(2d)^2-u_+^2=(1-a^2)(a^2+4d^2-1)>0$$
so $2d>u_+$. It follows that 
\begin{equation}
\label{eq:4d-wcr}
4d-w_+-\frac1{w_+}=\frac{2(2d-u_+)}{1-a^2}>0.
\end{equation}
Moreover, from \eqref{eq:Speqn} and from $w_+\in\mathbb{R}$  we see that $(4d-w_+-1/w_+)^2-4>0$. Thus the expression inside the logarithm in \eqref{eq:Sinden} is a positive real when $w=w_+$. 
\end{proof}

\subsubsection{Deformation of contours}
\label{subsubsec:defCont}
 
To compute the asymptotics of the contour integral \eqref{eq:PContInt} as $r\to\infty$, we will deform the contour of integration to go through the critical point $\wcr$ and in the vicinity of $\wcr$ match the steepest descent contour 
\[\tilde{C}=\{ z\in \C : \Im(S(z)) = \Im(S(\wcr))=0 \text{ and } \Re(S(z)) < \Re(S(\wcr)) \}.\]
The function $S(w)$ has singularities and branch cuts, and we next show that in the process of deforming the contour we do not cross any singularities or branch cuts. For the remainder of this section, we restrict to values of \(r\) and \(a\) so that both \(r\) and \(ar\) are positive integers.

First, a simple lemma. 
\begin{lem}\label{lem:sqroot-neg}
For \(d>1\) the equation \((4d-w-1/w)^2-4=0\) has 4 solutions, \(w_1<w_2<w_3<w_4\), where
\begin{align*}
w_1& =  2d +1 - 2 \sqrt{d(d+1)} \qquad & w_2 =  2d-1 - 2 \sqrt{d(d-1)}\\
w_3& = 2d-1 + 2 \sqrt{d(d-1)} \qquad &w_4 =  2d+1 + 2 \sqrt{d(d+1)}
\end{align*}
with \(w_1>0\), \(w_2<1\), $w_3>1$ and \(w_4> 3 + 2 \sqrt{2}\).
\end{lem}

\begin{proof}
The arguments are elementary.
\end{proof}

We now determine the locations of the singularities and branch cuts. 

\begin{lem}
\label{lem:branchcuts}
The branch cuts and singularities of \(S(w)\) occur in the following places:\begin{align*}
\mathcal{B}=&\{w:e^{S(w)}<0\} \cup (w_1,w_2)\cup (w_3,w_4) 
\\&\cup \left\{ x\pm i \sqrt{\frac{x}{4d-x}-x^2} : x \in (0,2d-\sqrt{4d^2-1})\cup(2d+\sqrt{4d^2-1},4d)
 \right\}
\end{align*}

Figure \ref{fig:contourPlots} shows plots of the branch cuts (in blue).

\begin{figure}
    \centering
        \includegraphics[width=11cm]{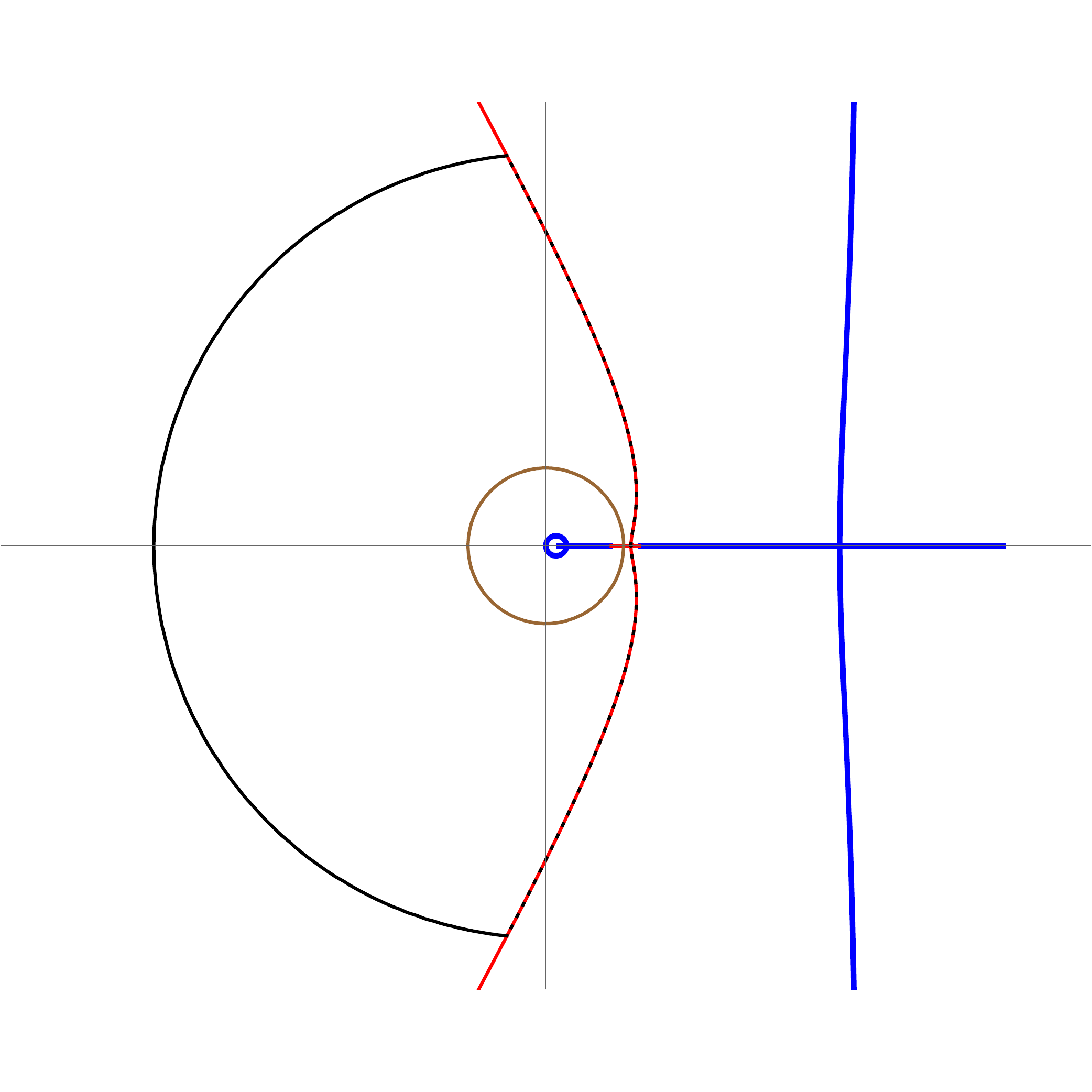}
        \caption{Plots of the branch cuts of $S(w)$ (blue), the contour $\Im(S(w))=0$ (red), the original contour (brown), the deformed contour $C'$ (the part along the steepest descent contour is red-and-black dashed, the circular part is black). The parameters are $a=0.5$ and $d=1.01$. }
        \label{fig:contourPlots}
\end{figure}
\end{lem}

\begin{proof}

There are branch cuts which come from the square root term in the denominator in \eqref{eq:Sinden}. These occur when \((4d-w-1/w)^2-4<0\) which, by Lemma \ref{lem:sqroot-neg}, means $w$ is in $(w_1,w_2)$ or $(w_3,w_4)$.  

We have several sources of branch cuts and singularities.

\begin{enumerate}
\item We have a singularity at \(w=0\).

\item By assumption both \(r\) and \(ar\) are positive integers and there is no branch cut when \(w<0\).

\item We have a branch cut from the square root term in \eqref{eq:Sinden}
when \((4d-w-1/w)^2-4<0\). In this case there are two possibilities. Either \(\Re(4d-w-1/w)=0\) or \(\Im(4d-w-1/w)=0\). 

First suppose \(\Im(4d-w-1/w)=0\). Then writing \(w=x+ i y\) for \((x,y)\in \R^2\) we have 
\begin{align*}
0& = \Im (4d-w-1/w) 
= \Im \left( 4d-x-i y - \frac{x-iy}{x^2+y^2} \right) 
 = y\left(-1 + \frac{1}{x^2+y^2} \right)
\end{align*}
which has 2 solutions, either \(y=0\) or \(x^2+y^2=1\). When $x^2+y^2=1$ we have $-2\leq w+1/w\leq 2$ so $$(4d-w-1/w)^2-4\geq (4d-2)^2-4>0$$ since $d>1$. It follows that we, in fact, don't have branch cuts along the unit circle. When $y=0$ Lemma \ref{lem:sqroot-neg} implies that the branch cuts occur when $w$ is in $(w_1,w_2)$ or $(w_3,w_4)$.  

Next, suppose \(\Re(4d-w-1/w)=0\). In this case we have
\begin{align*}
0&
= \Re \left( 4d-x-i y - \frac{x-iy}{x^2+y^2} \right)  
= 4d-x\left( 1 + \frac{1}{x^2+y^2} \right).
\end{align*}
Since $d>1$, it follows that $0<x<4d$ and, solving for $y$, we get that the branch cut lies on the curve $w=x\pm i\sqrt{ \frac{x}{4d-x}-x^2}$. Since $y$ must be real, we should have $ \frac{x}{4d-x}-x^2>0$, which implies $x<2d-\sqrt{4d^2-1}$ or $x>2d+\sqrt{4d^2-1}$. Combining with $0<x<4d$ we get that either 
$$0<x<2d-\sqrt{4d^2-1}$$
or
$$2d+\sqrt{4d^2-1}<x<4d.$$

\item We have a branch cut for the logarithm when $e^{S(w)}<0$. It will not be necessary for us to identify where exactly this occurs, so we do not study this case.

\end{enumerate}

\end{proof}

Our deformed contour will go through the critical point $\wcr$ and initially follow the steepest descent curve $\Im(S(w))=\Im(S(\wcr))=0$, where the last equality was established in Lemma \ref{lem:real-roots}. We now show that this contour does not approach the set $\mathcal{B}$ from Lemma \ref{lem:branchcuts}.

\begin{lem}
\label{lem:curveAvoidsB}
For any \(w\in \mathcal B\setminus \wcr\) we have \(\Im S(w)\neq 0\).
\end{lem}

\begin{proof}
In order for \(\Im S(w)=0\) we must have 
\begin{equation}
\label{eq:insideS}
\frac{4d-w-\frac{1}{w}-\sqrt{\left(4d-w-\frac{1}{w}\right)^2-4}}{2 w^a}>0.
\end{equation}
We consider cases based on the components of $\mathcal B$. 
\begin{enumerate}

\item Suppose $w\in\mathbb{R}$. If \(w<0\) then the numerator in \eqref{eq:insideS} is positive but the denominator is imaginary which gives a contradiction. If \(w>0\) then \eqref{eq:insideS} implies \((4d-w-1/w)^2-4>0\) so we're not on a branch cut.

\item If $w$ is such that $e^{S(w)}<0$ then $\Im(S(w))$ is $\pm \pi$. 

\item If \(w= x \pm i \sqrt{\frac{x}{4d-x}-x^2}\) with $x\in (0,2d-\sqrt{4d^2-1})\cup(2d+\sqrt{4d^2-1})$ then from the proof of Lemma \ref{lem:branchcuts} we have that $\Re(4d-w-1/w)=0$ so the numerator of \eqref{eq:insideS} is purely imaginary so in order for \eqref{eq:insideS} to hold we would need $w^a$ to be purely imaginary as well. Since $x>0$, the angular component of $w$ is less than $\pi/2$ and since $0\leq a\leq 1$ so is the angular component of $w^a$.

\end{enumerate}

\end{proof}

It follows from Lemma \ref{lem:real-roots} that if $w\in\mathbb{R}$ and is near the critical point $\wcr$ then $\Im(S(w))=0$. Thus the curve $\Im(S(w))=0$ passing through the critical point $\wcr$ has two components - an interval on the real line and a portion that leaves $\wcr$ in the vertical direction. We now compute the sign of $S''(\wcr)$ to show that the vertical direction is the steepest descent direction.

\begin{lem}
\label{lem:SppSign}
We have $S''(\wcr)>0$. 
\end{lem}
\begin{proof}
From \eqref{eq:Sp} we compute the second derivative at the critical point to be
\begin{align}\label{eq:spp}
S''(\wcr)  = - \left( \frac{z_+'' }{z_+}  - \frac{ z_+'^2}{z_+^2} - \frac{a}{ \wcr^2} \right) = - \frac{z_+'' }{z_+}  + \frac{a(1+a)}{ \wcr^2} ,
\end{align}
where the last equation was obtained from \eqref{eq:Sp}. From \eqref{eq:zpm} we have 
\begin{align}\label{eq:zpp}
z_+''(\wcr) & =\frac{-2}{\sqrt{\left(4d-\wcr-\frac 1\wcr\right)^2-4}}\left(  \frac{z_+(\wcr)}{\wcr^3}+ \frac{\left( -1+\wcr^{-2} \right)^2}{\left( 4d-\wcr -\frac 1\wcr \right)^2-4} \right).
\end{align}

Using \eqref{eq:4d-wcr} and \eqref{eq:zpm} we see that $z_+''(\wcr)<0$, whence $S''(\wcr)>0$. 
\end{proof}

It follows from Lemma \ref{lem:SppSign} that $\Re(S(w))$ increases when $w$ moves away from $\wcr$ along the real axis. Thus, $\Re(S(w))$ decreases when $w$ moves away from $\wcr$ along the vertical component of the curve $\Im(S(w))=0$, so that is the steepest descent curve.

When $|w|=1$ we have $4d-w-1/w>2$ so the denominator in \eqref{eq:Sinden} is a positive real so $\Im(S(w))\neq 0$ along the unit circle. This, combined with Lemma \ref{lem:curveAvoidsB} implies that the steepest descent curve avoids all the branch cuts and singularities, so must escape to infinity. Thus the steepest descent curve will have points $w$ on it with $|w|$ arbitrarily large.

Before we describe the deformed contour completely, two simple lemmas.

\begin{lem}
\label{lem:steepestCurveGoesNegative}
If $w$ is such that $\Im(S(w))=0$ and $|w|\gg 1$ then $\Re(w)<0$. 
\end{lem}
Figure \ref{fig:contourPlots} shows the contour $\Im(S(w))=0$ in red.
\begin{proof}
Suppose $\Re(w)\geq 0$. Then $\sqrt{w^2}=w$ so from \eqref{eq:Sinnum} we get
$$S(w)=\log\left(\frac{-w-\sqrt{w^2}+O(1)}{2w^a}\right)=\log\left(-w^{1-a}+o(w^{1-a})\right).$$
Since $\Im(S(w))=0$, we must have $\Im(-w^{1-a})$ should be small, which implies $w^{1-a}$ must be close to the negative real line. Since $1-a\geq 0$, that cannot happen when $\Re(w)\geq 0$. 
\end{proof}

\begin{lem}
\label{lem:ReSsmall}
If $w$ is such that $|w|\gg 1$ and $\Re(w)<0$ then $\Re(S(w))\ll 0$. 
\end{lem}
\begin{proof}
From \eqref{eq:Sinden} it is clear that it suffices to do the case $a=0$. In that case we have
\begin{align*}
\Re(S(w))=\log\left(\left|\frac{2}{-w+O(1)+\sqrt{w^2+O(1)}}\right|\right)
\end{align*}
Since $\Re(w)<0$ we have $w=v e^{i \theta}$ with $v>0$ and $\frac{\pi}2<\theta\leq \pi$. It follows that $\sqrt{w^2}=v e^{-i \theta}$ and we get
\begin{align*}
\Re(S(w))=\log\left(\frac{2}{\left|v(e^{-i\theta}-e^{i\theta})+o(v)\right|}\right)
=-\log\left(\left|v\sin\theta+o(v)\right|\right)\ll 0.
\end{align*}

\end{proof}

We deform the integration contour $C$ in \eqref{eq:PContInt} to the contour $C'$ which is obtained as follows. We start at the critical point $\wcr$ and follow the vertical branch of the curve $\Im(S(w))=0$. By the discussion before Lemma \ref{lem:steepestCurveGoesNegative} we follow this curve until $|w|\gg 1$ at which point by Lemma \ref{lem:steepestCurveGoesNegative} we have $\Re(w)<0$. From that point we follow the circle centered at the origin until we hit negative real line. The rest of $C'$ is the reflection of the first part with respect to the real line. Figure \ref{fig:contourPlots} shows plots of the original and deformed contours. 

To show that the visited region is within a constant of the scaled limit curve we need the following lemma.

\begin{lemma}\label{lem:coefficient inequalities}
    For each \(d>1\), there exist constants \(0<b_d<B_d\) that depend on only on $d$ such that
    \begin{align*}
        b_d \leq -\scr  \leq B_d,  \qquad b_d \leq S''(\wcr)  \leq B_d,  \qquad \text{and} \qquad b_d \leq G(\wcr)  \leq B_d.
    \end{align*}
\end{lemma}

\begin{proof}
    From the proof of Lemma \ref{lem:real-roots} we have that \(\wcr\) is an increasing function of \(a\) with \(\wcr(0)  = 1\) and \(\wcr(1) = d + \sqrt{d^2-1}\), which gives the inequality \(1 \leq \wcr \leq d + \sqrt{d^2-1}\).
    Thus we have 
    \begin{align}\label{eq:wcrBdd}
        2d&\leq 4d-\wcr - 1/\wcr \leq 4d-2.
    \end{align}
    Furthermore, \((4d-w - 1/w)^2- 4\) is a decreasing function of \(w\) in the range \(1\leq w\leq 2d+\sqrt{4d^2-1}\). Since \(1\leq \wcr\leq d + \sqrt{d^2-1}<2d+\sqrt{4d^2-1}\) we find
    \begin{align*}
       4(d^2-1)&< (4d-\wcr - 1/\wcr)^2- 4 < 16d(d-1).
    \end{align*}
    Inserting these inequalities into \eqref{eq:G} gives
    \begin{align*}
        \frac{1}{(d+\sqrt{d^2-1}) \sqrt{16d(d-1)}}& \leq G(\wcr)  \leq \frac{1}{\sqrt{4(d^2-1)}}.
    \end{align*}

Next, from \eqref{eq:zpm} and \eqref{eq:wcrBdd} it is clear that 
\begin{equation}\label{eq:zpBdd}
1<d+\sqrt{d^2-1}\leq z_+(\wcr)\leq 2d-1+2\sqrt{d(d-1)}.
\end{equation}
Thus we obtain the desired bounds for \(S(w)\) which is given by \eqref{eq:Sdeff}. In particular, we have 
    \begin{align*}
\scr  \leq -\log\left( d+\sqrt{d^2-1} \right)<0.
    \end{align*}

    Last, from \eqref{eq:spp} and the bounds \eqref{eq:wcrBdd} and \eqref{eq:zpBdd} it is clear that to show \(S''(\wcr)\) is bounded, away from zero, it is enough to show that $z_+''$ is bounded by constants that depend on $d$ only. This is immediate from \eqref{eq:zpp}, \eqref{eq:wcrBdd} and \eqref{eq:zpBdd}.

\end{proof}

\subsubsection{Proof of the first main theorem}
\begin{proof}[Proof of Theorem \ref{thm:main}]

From Lemma \ref{lem:branchcuts} it follows that when deforming the contour of integration in \eqref{eq:PContInt} from $C$ to $C'$ we do not pick up any residues. From the discussion before Lemma \ref{lem:steepestCurveGoesNegative} and from Lemma \ref{lem:ReSsmall} we obtain that along the contour $C'$ we have that $\Re(S(w))$ is maximized at the critical point $\wcr$. Thus, as $r\to\infty$, the  contribution to the integral in \eqref{eq:PContInt} from points of $C'$ away from the critical point is going to be exponentially smaller than the contribution from the vicinity of $\wcr$. 

Making the change of variable $w=\wcr+i\frac{y}{\sqrt{r}}$ we obtain 
\begin{align}
\nonumber
P_d(r,ar)&=\frac{c(d-1)}{2\pi\sqrt{r}}G(\wcr)e^{rS(\wcr)}\int_{-\infty}^{\infty}e^{-\frac{S''(\wcr)y^2}{2}}(1+o(1))dy.
\\&\label{eq:PdAsympt}
=\frac{c(d-1)}{\sqrt{2\pi S''(\wcr)r}}G(\wcr)e^{rS(\wcr)}(1+o(1)).
\end{align}

Based on Lemma \ref{lem:prob-limit-shape}, given a direction $0\leq a\leq 1$ define the outer and inner radii $r_o$ and $r_i$ by 
\begin{align}
\label{eq:r_io}
P_d(r_o,ar_o)=\frac{c(d-1)}{n}\text{ and }P_d(r_i,ar_i)=\frac{cd}{n}.
\end{align}

Solving these equations using \eqref{eq:PdAsympt} we obtain
\begin{align}
    \label{eq:r_oasympt}
    r_o & \leq \frac{-1}{\scr}\left( \log n -\frac{1}{2}\log \log n - \log \left( \frac{1}{G(\wcr)} \sqrt{\frac{2\pi S''(\wcr)}{-\scr}} \right)\right) + q_o,\\
    \label{eq:r_iasympt}
    r_i & \geq \frac{-1}{\scr}\left( \log n -\frac{1}{2}\log \log n - \log \left( \frac{1}{G(\wcr)} \sqrt{\frac{2\pi S''(\wcr)}{-\scr}} \frac{d}{d-1}\right) \right) + q_i,
\end{align}
where \(q_o,q_i \xrightarrow[]{n\to \infty} 0\). 


Since $\lim_{n\to\infty} \frac{r_o}{r_i}=1$ it follows from Lemma \ref{lem:prob-limit-shape} that if we scale by $\log n$ the region visited by the Leaky-ASM started with $n$ chips at the origin, it will converge to the parametric curve 
\begin{align}
\label{eq:limitShapeScr}
-\left(  \frac{1}{S(\wcr)},  \frac{a}{S(\wcr)} \right) \text{ for } 0\leq a\leq 1,
\end{align}
and its reflections with respect to the coordinate axes and the diagonal $y=x$.
This curve is in fact the dual of the gaseous component of the amoeba of \(P(z,w)\). See Section \ref{sec:amoebae} for more details.

It follows from Lemma \ref{lem:coefficient inequalities}, that for each \(d>1\) there exists a constant $k_d>0$ depending on $d$ only, such that
\begin{align*}
    r_o & \leq \frac{-1}{\scr}\left( \log n -\frac{1}{2}\log \log n\right) + k_d + q_o,\\
    r_i & \geq \frac{-1}{\scr}\left( \log n -\frac{1}{2}\log \log n)\right) - k_d + q_i,
\end{align*}
Thus if we scale the limiting curve by \(\log n -\frac{1}{2}\log \log n\), then for each direction specified by \(a\) the visited region is within a constant depending only on \(d\) from the scaled curve if $n$ is large enough.

\end{proof}

\begin{remark}
It is straightforward to verify that as $d\to 1$ we have  
\begin{align*}
    S(\wcr)&=-2\sqrt{a^2+1}\sqrt{d-1}(1+o(1)).
\end{align*}
It follows that the curve becomes
\begin{align*}
\frac{1}{2\sqrt{d-1}}\left(\frac1{\sqrt{a^2+1}},\frac a{\sqrt{a^2+1}}\right)(1+o(1)),
\end{align*}
so it converges to a circle. This is further studied in the next section. See Figures \ref{fig:4dplots} and \ref{fig:dsmall}.

On the other hand, as $d\to \infty$ we have  
    \begin{align*}
        S(\wcr)&=-(1+a)\log(d)(1+o(1)).
    \end{align*}
It follows that the curve becomes
\begin{align*}
\frac{1}{\log d}\left(\frac1{1+a},\frac a{1+a}\right)(1+o(1)),
\end{align*}
so it converges to the $L^1$ ball. See Figures \ref{fig:4dplots} and \ref{fig:dlarge}. The heuristic explanation why the $L^1$ ball should appear in this limit is the same as the one given in Remark \ref{rem:limitTriangle} for the 2-directional version of the model. 

\end{remark}

\begin{remark}
Note, that since we obtain the limit shape by reflecting a smooth curve with respect to various lines, a priori the limit shape might not be differentiable along the axes of symmetry, however it is easy to verify that the curve \eqref{eq:limitShapeScr} has slope infinity as $a\downarrow 0$ and slope $-1$ as $a\uparrow 1$, so the reflections of it result in a differentiable curve.
\end{remark}

\section{The Limit Shape of the Leaky-ASM When the Leakiness Vanishes}
\label{sec:LeakTo0}
In this section we give a proof of Theorem \ref{thm:leakTo0}. We let \(d=1+t_n\) and assume that \(t_n\to 0\) as $n\to\infty$. Note, that Lemma \ref{lem:prob-limit-shape} still applies, so to compute the limit shape we need to compute the asymptotic behavior of \(P_d(r_n, ar_n)\) as \(r_n\to\infty\) simultaneously with \(t_n \to 0\). We can assume we are in the regime \(r_n \sqrt{t_n}\to\infty\).		

As in the previous section we will work with the contour-integral representation of $P_d$ given by \eqref{eq:PContInt} where $G$ is given in \eqref{eq:G} and $S$ in \eqref{eq:Sinnum}. As before let \(\wcr\) be the larger critical point of \(S(w)\) given by \eqref{eq:wcr}. A series expansion gives 
\begin{align}
\label{eq:wcrSeries}
    \wcr & = 1+\frac{2 a \sqrt{t_n}}{\sqrt{1+a^2}}+\frac{2 a^2 t_n}{1+a^2}+O(t_n^{3/2}).
\end{align}
We deform the integration contour in the same way as in Section \ref{sec:4dshape}. The analysis is exactly the same. The only difference is in the main contribution coming from the vicinity of the critical point which we now study.

\begin{prop}
\label{prop:asym-expansions}
    Let \(\beta_n=\sqrt{\frac{\sqrt{t_n}}{r_n}}\). In the regime when $n\to\infty$, $r_n\to\infty$, $t_n\to 0$ with $r_n\sqrt{t_n}\to\infty$ we have the expansions
    \begin{align}\label{eq:asym-expansionS}
        r_n\left( S(\wcr+i\beta_n y) - S(\wcr) \right)&=-y^2\frac{(1+a^2)^{\frac 32}}{2}+o(1),
        \\ \label{eq:asym-expansionG}
        G(\wcr+i\beta_ny)&=G(\wcr)+o(1)=\frac{\sqrt{1+a^2}}{4\sqrt{t_n}}+o(1).
    \end{align}
\end{prop}

\begin{proof}
    We have
    \begin{align}
\label{eq:S-Scr}
        r_n\left( S(\wcr + i \beta_n y) - S(\wcr) \right) & =-r_n \log \left( \frac{z_+(\wcr + i \beta_n y)}{z_+(\wcr)} \right) - a r_n \log \left( \frac{\wcr + i \beta_n y}{\wcr} \right),
    \end{align}
where $z_+(w)$ is defined in \eqref{eq:zpm}.
    We expand each term separately in a series to identify the leading order terms as \(n\to\infty\). For notational purposes let \(v(w)=4(1+t_n)-w-1/w\) and $\vcr=v(\wcr)$, so that 
    \begin{align*}
        z_+(w)&
        =\frac{v+\sqrt{v^2-4}}{2}.
    \end{align*}
    From \eqref{eq:wcrSeries} we get that
    \begin{align}
\label{eq:vcr}
        \vcr&=2+\frac{4 t_n}{1+a^2}+O(t_n^{3/2})
    \end{align}
    and 
    \begin{align*}
        v(\wcr+i\beta_n y)
        &=\vcr-i\beta_n y\left(1-\wcr^{-2}\right)+\beta_n^2 y^2+O(\beta_n^2\sqrt{t_n}),
    \end{align*}
where we used that $\beta_n\ll \sqrt{t_n}$, which follows from $r_n\sqrt{t_n}\to\infty$. It follows that
    \begin{align*}
        &\sqrt{v(\wcr+i\beta_ny)^2-4}
        =\sqrt{\vcr^2-4-2 i y\beta _n \vcr(1-\wcr^{-2})+4\beta_n^2 y^2+O(\beta_n^2\sqrt{t_n})},
    \end{align*}
where we used that
    \begin{align*}
        1-\wcr^{-2}&=\frac{4a\sqrt{t_n}}{\sqrt{1+a^{2}}}+O(t_n),
    \end{align*}
which follows from \eqref{eq:wcrSeries}. 
It follows from \eqref{eq:vcr} that
    \begin{align*}
        \vcr^2-4&=\frac{16t_n}{1+a^2}+O(t_n^2)
    \end{align*}
so $\vcr^2-4$ is the dominant order term in the square root above. Factoring out the dominant term and expanding the square root gives
    \begin{multline}
    \label{eq:v^2-4}
        \sqrt{v(\wcr+i\beta_ny)^2-4}
\\
        =\sqrt{\vcr^2-4}\left( 1-i y\frac{\beta_n \vcr(1-\wcr^{-2})}{\vcr^2-4} +
        y^2\frac{\beta_n^2\vcr}{\vcr^2-4}\left(1+\frac{\vcr^2(1-\wcr^{-2})^2}{4(\vcr^2-4)}\right)+
        O\left(\frac{\beta_n^2}{\sqrt{t_n}}\right)\right).
\\
        =\sqrt{\vcr^2-4}-i y\frac{\beta_n \vcr(1-\wcr^{-2})}{\sqrt{\vcr^2-4}} +
        y^2\frac{(1+a^2)^{\frac 32}}{2}\frac{\beta_n^2}{\sqrt{t_n}}+
        O(\beta_n^2).
    \end{multline}
    Inserting this series into the equation for \(z_+\) gives
    \begin{align*}
        2z_{+}(\wcr+i\beta_ny)
        &=\left(\vcr+\sqrt{\vcr^2-4}\right)\left(1-\frac{i\beta_n y(1-\wcr^{-2})}{\sqrt{\vcr^2-4}} \right)
        +y^2\frac{(1+a^2)^{\frac 32}}{2}\frac{\beta_n^2}{\sqrt{t_n}}
        +O(\beta_n^2).
    \end{align*}
    It follows that
    \begin{align*}
        \frac{z_+(\wcr+i\beta_n y)}{z_+(\wcr)}& =1-\frac{i\beta_n y(1-\wcr^{-2})}{\sqrt{\vcr^2-4}}+y^2\frac{(1+a^2)^{\frac 32}}{2z_+(\wcr)}\frac{\beta_n^2}{\sqrt{t_n}}+O(\beta_n^2).
    \end{align*}
 Using $z_+(\wcr)=1+O(\sqrt{t_n})$, the definition of $\beta_n$ and the fact that at a critical point \eqref{eq:Speqn} holds, we get
    \begin{align}
\label{eq:zpPart}
        \log\left(\frac{z_+(\wcr+i\beta_n y)}{z_+(\wcr)}\right)
        =-iy\frac{a\beta_n}{\wcr}+y^2\frac{(1+a^2)^{\frac 32}}{2r_n}+O\left(\frac{\sqrt{t_n}}{r_n}\right).
    \end{align}

    Expanding the second term in \eqref{eq:S-Scr} gives
    \begin{align*}
        \log \left( \frac{\wcr + i \beta_n y}{\wcr} \right)& 
         = iy\frac{\beta_n}{\wcr}+O\left( \beta_n^2 \right).
    \end{align*}
    Combining this with \eqref{eq:zpPart}, the expression \eqref{eq:S-Scr} simplifies to \eqref{eq:asym-expansionS}.

    For the expansion of \(G(\wcr+ i \beta_ny) \) we use \eqref{eq:v^2-4} to obtain
    \begin{align*}
        \frac{1}{G(\wcr+i\beta_n y)}&=\left( \wcr+i\beta_n y \right)\sqrt{v(\wcr+i\beta_n y)^2-4}\\
        &=(\wcr+O(\beta_n))\left(\sqrt{\vcr^2-4}+O(\beta_n)\right)\\
        &=\frac{4\sqrt{t_n}}{\sqrt{1+a^2}}+O(\beta_n)
        =\frac{4\sqrt{t_n}}{\sqrt{1+a^2}}(1+o(1)).
    \end{align*}

\end{proof}

We are now ready to give a proof of the second main theorem. 
\begin{proof}[Proof of Theorem \ref{thm:leakTo0}]
Deforming the contour of integration in \eqref{eq:PContInt} as described in section \ref{subsubsec:defCont}, ignoring the exponentially smaller contribution of the portion of the deformed contour away from the critical point, rescaling $w$ as $w=\wcr+i\beta_n y$ near the critical point and using the asymptotics obtained in Proposition \ref{prop:asym-expansions}, we have
\begin{align*}
    P_d(r_n, ar_n)
    & = \frac{ct_n}{2 \pi i} e^{r_nS(\wcr)} \oint_{C'} \left(G(\wcr)+o(1)\right) e^{ r_n (S(w)-S(\wcr))} dw\\
    &=\frac{ct_n \beta_n}{2 \pi} e^{r_nS(\wcr)} G(\wcr) \int_{-\infty}^{\infty} e^{-\frac{y^2(1+a^2)^{\frac 32}}{2}}\left( 1+o(1) \right)dy.
\end{align*}
Using \eqref{eq:asym-expansionG} and the expansion
\begin{align*}
    S(\wcr)=-2\sqrt{(1+a^2)t_n} + o(\sqrt{t_n})
\end{align*}
we get
\begin{align}
\nonumber
    P_d(r_n, ar_n)
    &=\frac{ct_n \beta_n}{\sqrt{2\pi}(1+a^2)^{\frac 34}}\frac{\sqrt{1+a^2}}{4\sqrt{t_n}}e^{-2r_n\sqrt{(1+a^2)t_n}}\left( 1+o(1) \right)\\
\label{eq:PdAsymptDto1}
    &=\frac{\gamma t_n}{\sqrt{r_n\sqrt{t_n}}}e^{-2r_n\sqrt{(1+a^2)t_n}}\left( 1+o(1) \right),
\end{align}
where the constant \(\gamma=\frac{c}{4\sqrt{2\pi}\left( 1+a^2 \right)^{1/4}}\). 

As in \eqref{eq:r_io}, for a direction \(0\leq a\leq 1\) we define the inner and outer radii, \(r_o\) and \(r_i\), by 
\begin{align*}
    P_d(r_o,ar_o) & = \frac{c(d-1)}{n} =\frac{c t_n}{ n},\\
    P_d(r_i,ar_i) &= \frac{cd}{n} = \frac{c(1+t_n)}{n}.
\end{align*}
Comparing these to \eqref{eq:r_oasympt},\eqref{eq:r_iasympt} and \eqref{eq:PdAsymptDto1} to \eqref{eq:PdAsympt} we obtain
\begin{align*}
    r_o\sqrt{t_n} &= \frac{1}{2\sqrt{(1+a^{2})}}\log( n)\left( 1+o(1) \right)\\
    r_i\sqrt{t_n} &= \frac{1}{2\sqrt{(1+a^{2})}}\log\left( n t_n \right)\left( 1+o(1) \right),
\end{align*}
where the last one is valid as long as $nt_n\to\infty$. 
Taking ratios of the inner and outer curves we get
\begin{align*}
    \frac{r_i}{r_o}&=\left(1+\frac{\log(t_n)}{\log(n)}\right)\left( 1+o(1) \right)
\end{align*}
Using Lemma \ref{lem:prob-limit-shape} we conclude that if $t_n$ decays slower than any power of $n$, say like $\log(n)$, then $r_i/r_o\to 1$, so we have that the boundary of the visited region of the Leaky-ASM, scaled down by $\log(n)/\sqrt{t_n}$ converges to a circle. 

If $t_n$ converges to $0$ faster, say $t_n=\frac{1}{n^{1-\alpha}}$, 
then \(\frac{r_i}{r_o}\to\alpha\) so after scaling we can only bound the boundary of the visited region between between two circles of different radii.
\end{proof}

\section{Limit Shapes and Amoebae}
\label{sec:amoebae}

In this section we show that the limit shapes we obtain are in fact certain portions of the duals of the amoebae of the corresponding spectral curves. The connection between limit shapes and amoebae in similar models is known to experts in the field, but we present the details in our case for completeness. See, for example, \cite{kos2006} for a discussion of amoebae in dimer models. A more general survey of mathematical results related to amoebae can be found in \cite{m2004}. 

Given a polynomial \(P(z,w)\) let \(V=\{(z,w)\in\C^2: P(z,w)=0\}\) be its \emph{spectral curve}. The \emph{amoeba} of \(P(z,w)\) denoted by \(\mathbb{A}(P)\) is the image of the spectral curve in \(\C^2\) under the map 
\[(z,w) \mapsto (\log|z|, \log|w|).\]
The bounded component of the complement of the amoeba is known as its \emph{gaseous phase} (see \cite{kos2006}).

In this section we show that the curve \eqref{eq:limitShapeScr} is the dual of the gaseous phase of the amoeba of \(P(z,w)\) given by \eqref{eq:P-uniform}. Figure \ref{fig:amoeba} shows a plot of the amoeba and its dual when $d=2$. 

\begin{figure}
    \centering
        \includegraphics[width=9cm]{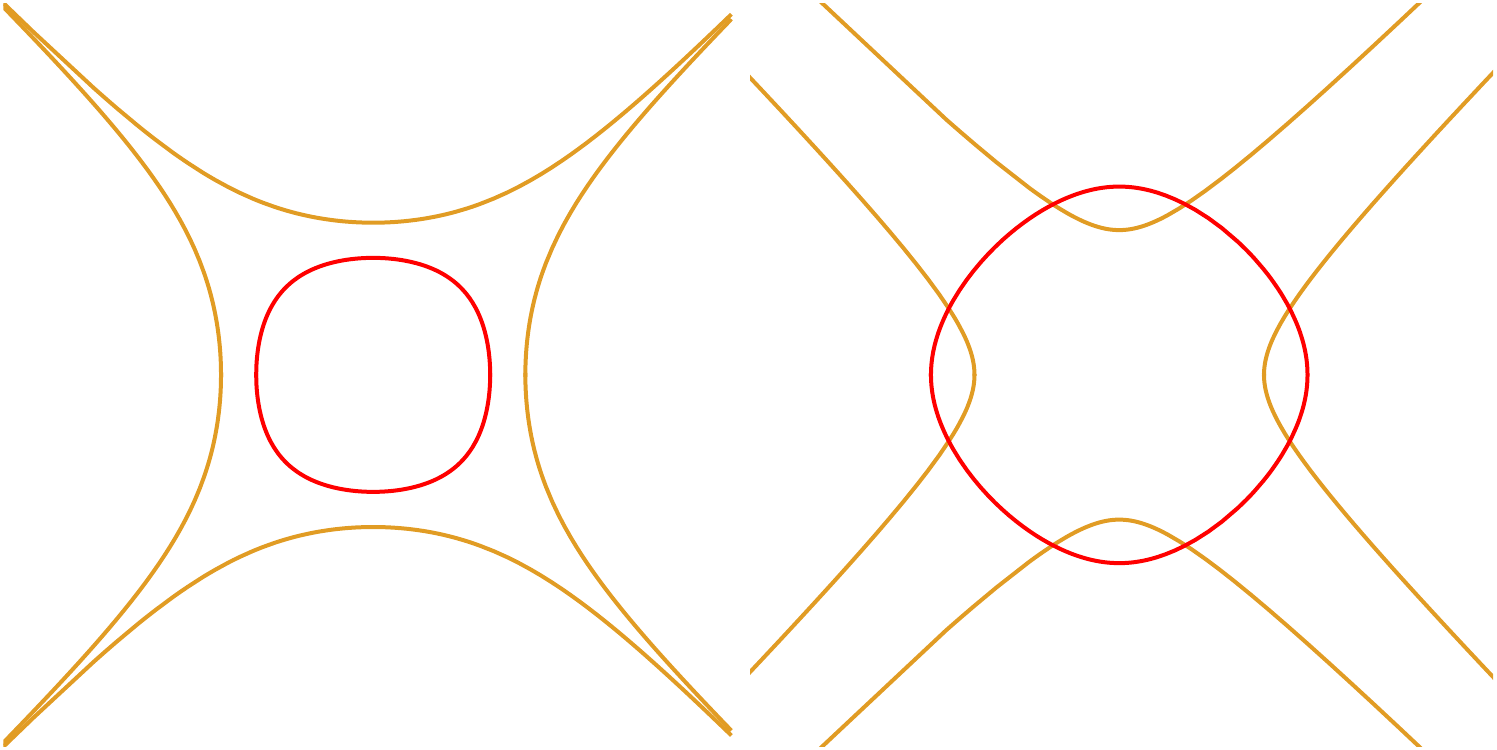}
        \caption{On the left we have the boundary of the amoeba of \(P(z,w)=\frac{4d-z-\frac{1}{z}-w-\frac{1}{w}}{4(d-1)}\) with $d=2$. The boundary of the the gaseous component is in red. On the right we have the dual curve, with the portion corresponding to the gas boundary, i.e. the limit shape in our model, in red.}
        \label{fig:amoeba}
\end{figure}

Abusing notation, let \(S(z,w)= -\ln(z)-a\ln(w).\) We have $S(w)=S(z_+(w),w)$, where $z_+(w)$ is the solution of $P(z_+(w),w)=0$ given by \eqref{eq:zpm}. 

For a parametric curve \((x(t),y(t))\) its dual curve can be parametrized as $$\frac{1}{Q(t)}\left(y'(t),-x'(t) \right)$$ where 
\[Q(t) = y(t) x'(t)-x(t)y'(t).\]
Applying this to the curve \eqref{eq:limitShapeScr} we obtain
\begin{align*}
Q(a)=-\frac{a}{S}\frac{S_a}{S^2}+\frac{1}{S}\left(\frac{aS_a}{S^2}-\frac 1S\right)=-\frac1{S^2},
\end{align*}
where we used $S$ for $S(z_+(\wcr),\wcr)$ and $S_a$ for $dS(z_+(\wcr),\wcr)/da$. For the parametrized dual curve we obtain
\begin{align*}
(S-aS_a,S_a).
\end{align*}
Computing the derivative $S_a$ we get
\begin{align*}
S_a&=\frac{\partial S}{\partial a}  +\frac{\partial S}{\partial z} (z_+(\wcr),\wcr)\frac{d z_+(\wcr)}{ d a} + \frac{\partial S}{\partial w}(z_+(\wcr),\wcr) \frac{d \wcr}{d a} \\
    &= -\ln(\wcr) -\left(\frac{1}{z_+} \frac{d z_+}{d w} +\frac{a}{w}\right)\bigg|_{w=\wcr} \frac{d \wcr}{d a}.
\end{align*}
From the definition of the critical point $\wcr$ we have
\begin{align*}
0&=\frac{d}{dw}S(z_+(w),w) \bigg |_{w=\wcr}= \left(-\frac{1}{z_+}\frac{dz_+}{dw}-\frac{a}{w}\right)\bigg|_{w=\wcr},
\end{align*}
which implies that $S_a=-\log(\wcr)$. It follows that the dual curve is 
\begin{equation}
\label{eq:amoeba-reparam}
-(\ln(z_+(\wcr)), \ln(\wcr)).
\end{equation}
From \eqref{eq:P} it is evident that boundary of the bounded region of the complement of the amoeba of $P(z,w)$ is the curve $(\log(z),\log(w))$ with $z,w>0, P(z,w)=0$. 
As we saw in Section \ref{sec:4dshape} $\wcr$ and $z_+(\wcr)$ are positive, so the curve \eqref{eq:amoeba-reparam} is nothing but a reparametrization (of the reflection with respect to the origin) of the boundary of the gaseous phase of the amoeba. 

Recall, that in the 2-directional case studied in Section \ref{sec:2dShape} the spectral polynomial $P$ is $P(z,w)=\frac{2d-z-w}{2(d-1)}$. Its Newton polygon (see \cite{kos2006} for a definition) does not have integer points in the interior, so its amoeba does not have a gaseous phase. In fact the boundary of the amoeba consists of the curves given implicitly by $e^x+e^y=2d$, $-e^x+e^y=2d$ and $e^x-e^y=2d$. In this case the limit shape of the Leaky-ASM is (the reflection with respect to the origin) of the dual of the portion of the amoeba given by $e^x+e^y=2d$. See Figure \ref{fig:2damoeba}.

\begin{figure}
    \centering
        \includegraphics[width=9cm]{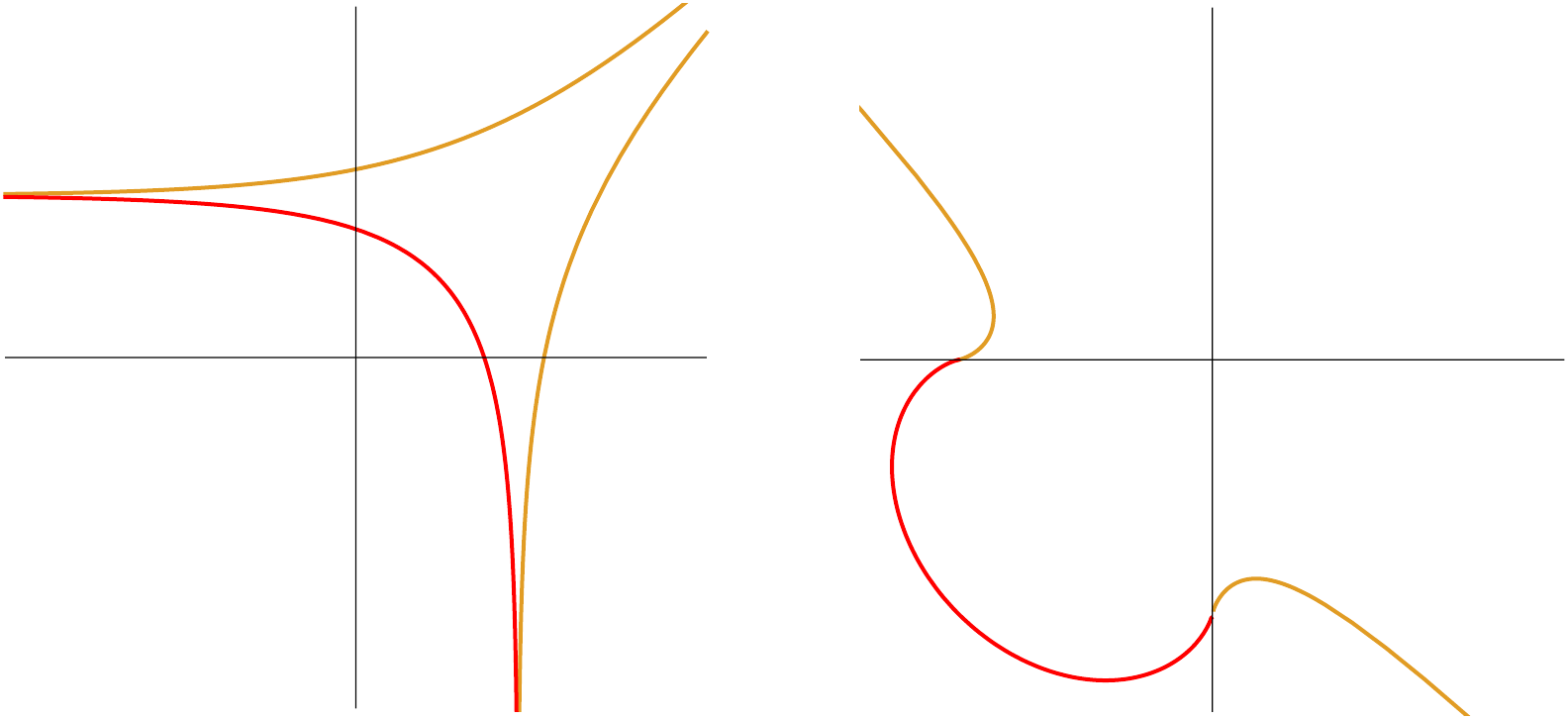}
        \caption{On the left we have the boundary of the amoeba of \(P(z,w)=\frac{2d-z-w}{2(d-1)}\) with $d=2$. On the right we have the dual curve. The red portions are the limit shape (in the plot on the right) and the corresponding section of the boundary of the amoeba (in the plot on the left).}
        \label{fig:2damoeba}
\end{figure}

\subsection*{Acknowledgments}   
We would like to thank Rick Kenyon for suggesting the study of the Leaky-ASM model and for useful discussions on the topic. We would also like to thank Arjun Krishnan and Henrik Shahgholian for useful discussions on sandpile models. Finally, we thank the anonymous referee for carefully reading the paper and suggesting a simplification to the proof of Lemma \ref{lem:prob-limit-shape}.

The second listed author was partially supported by the Simons Foundation Collaboration Grant No. 422190.

\bibliographystyle{hep}
\newcommand{\etalchar}[1]{$^{#1}$}

\end{document}